\documentclass{amsart}
\usepackage[utf8]{inputenc}
\usepackage{amsthm}
\usepackage{amssymb,amsmath}
\usepackage{xcolor}
\usepackage{tikz}
\usepackage{tikz-cd}
\usepackage{siunitx}
\usetikzlibrary{matrix, positioning, calc}
\usepackage{hyperref}
\usepackage[shortlabels]{enumitem}

\newtheorem{thm}{Theorem}[section]
\newtheorem{lem}[thm]{Lemma}
\newtheorem{defin}[thm]{Definition}
\newtheorem{prop}[thm]{Proposition}
\newtheorem{cor}[thm]{Corollary}
\newtheorem{rem}[thm]{Remark}
\newtheorem{claim}[thm]{Claim}

\newtheorem{conv}[thm]{Convention}

\newcommand {\ignore}[1]  {}

\newcommand{\N}{\mathcal{N}}

\newcommand{\Rel}{\mathrm{Rel}}
\newcommand{\M}{{\mathcal{M}}}

\newcommand{\BB}{\boldsymbol{B}}
\newcommand{\B}{B}
\newcommand{\area}{\mathrm{area}}
\newcommand{\bq}{\mathbf{q}}
\newcommand{\hol}{\mathrm{hol}}
\newcommand{\sm}{\smallsetminus}
\newcommand{\dev}{\mathrm{dev}}
\newcommand{\HH}{{\mathcal{H}}}
\newcommand{\R}{{\mathbb{R}}}
\newcommand{\vre}{{\varepsilon}}
\newcommand{\Res}{{\mathrm{Res}}}
\newcommand{\C}{{\mathbb{C}}}

\newcommand{\Mod}{{\operatorname{Mod}}}

\newcommand{\LL}{{\boldsymbol{L}}}
\newcommand{\x}{\mathrm{x}}
\newcommand{\y}{\mathrm{y}}

\newcommand{\product}{U'_{\x} \times U_{\y}}
\newcommand{\WW}{W_{L}}
\newcommand{\xform}{\beta_{\x}}
\newcommand{\xmeasure}{\nu_{\beta_\x}}

\newcommand{\q}{\boldsymbol{q}}
\newcommand{\df}{\overset{\operatorname{def}}{=}}
\newcommand{\HHm}{\mathcal{H}_{\mathrm{m}}}

\newcommand{\cone}{m_{\M}}
\newcommand{\GL}{\mathrm{GL}_2^+(\R)}
\newcommand{\SL}{\mathrm{SL}_2(\R)}
\newcommand{\U}{\mathcal{U}}
\newcommand{\V}{\mathcal{V}}
\newcommand{\W}{\mathcal{W}}

\title{Horospherical dynamics in invariant subvarieties}
\author{John Smillie}
\address{University of Warwick, Coventry, UK 
{\tt j.smillie@warwick.ac.uk}}
\author{Peter Smillie}
\address{Universit\"at Heidelberg, Heidelberg, Germany
{\tt psmillie@mathi.uni-heidelberg.de}}
\author{Barak Weiss}
\address{Dept. of Mathematics, Tel Aviv University, Tel Aviv, Israel
{\tt barakw@tauex.tau.ac.il}}
\author{Florent Ygouf}
\address{Dept. of Mathematics, Tel Aviv University, Tel Aviv, Israel
{\tt florentygouf@mail.tau.ac.il}}

\begin{document}

\maketitle
\begin{abstract}
 We consider the horospherical foliation on any invariant subvariety in the moduli space of translation surfaces. This foliation can be described dynamically as the strong unstable foliation for the geodesic flow on the invariant subvariety, and geometrically, it is induced by the canonical splitting of $\C$-valued cohomology into its real and imaginary parts. We define a natural volume form on the leaves of this foliation, and define horospherical measures as those measures whose conditional measures on leaves are given by the volume form. We show that the natural measures on invariant subvarieties, and in particular, the Masur-Veech measures on strata, are horospherical. We show that these measures are the unique horospherical measures giving zero mass to the set of surfaces with horizontal saddle connections, extending work of Lindenstrauss-Mirzakhani 
and Hamenst\"adt 
for principal strata. We describe all the leaf closures for the horospherical foliation.
\end{abstract}

\section{Introduction}

It is an interesting fact that geometric questions about rational polygonal billiards can be addressed by studying the dynamics on moduli spaces of translation surfaces. This is one of many reasons  to study the dynamics on moduli spaces of translation surfaces --- see 
 the surveys  \cite{MT, Zorich_survey, Forni-Matheus_survey, Wright_survey} for other motivation and a survey of results. We remind the reader that this moduli space is partitioned into {\em strata}, which correspond to translation surfaces of a fixed topological type. The group $G \df \SL$ acts on each stratum.

The {\em horocycle flow} is given by 
$$
U \df \left\{u_s  : s \in \R\right\} \subset G, \ \ \text{ where } \
u_s
 \df \left(\begin{matrix}  1 & s \\  0 & 1 \end{matrix}
  \right).
$$
The analogy between dynamics on strata and homogeneous dynamics has been fruitful. In the setting of homogeneous dynamics $U$-actions and $G$-actions were analyzed in work of Ratner which showed that orbit closures and ergodic invariant probability measures are surprisingly well-behaved. The dynamics of $G$-actions (and moreover the dynamics of its subgroup $P$ of upper triangular matrices) on strata were analyzed in  two papers \cite{EM, EMM} where it was shown that orbit closures and ergodic invariant measures have nice  descriptions (see Section~\ref{subsec: invariant subvarieties} for a precise statement). The situation for the $U$-action on the strata of the moduli spaces is now known to be more complicated due to the work of Chaika-Smillie-Weiss \cite{chaika2020tremors}.  

The $G$-orbit closures
are endowed with a wealth of geometrical structures, among which is the {\em horospherical foliation} which plays the role of the strong unstable manifold foliation for the one parameter diagonal subgroup which is called the geodesic flow (see \S \ref{subsec: horospherical foliation}). In \S \ref{sec: horospherical measures} we will define {\em horospherical measures}. Loosely speaking, the horospherical leaves are endowed with affine structures and the horospherical measures are those for which the conditional measures on theses leaves are translation invariant with respect to these affine structures. In the setting of homogeneous dynamics, there is a corresponding notion of horospherical dynamics. It has been established by Dani in \cite{dani1978invariant} and \cite{dani1981invariant} about a decade prior to the work of Ratner that these dynamical systems are also well-behaved. This paper is concerned with showing that horospherical measures and horospherical leaves in strata are also well-behaved.  


\subsection{Statement of results}
All measures considered in this paper are Borel regular Radon measures on strata of translation surfaces. Any $G$-orbit closure $\M^{(1)} \subset \HH^{(1)}$ supports a unique ergodic $G$-invariant finite smooth measure; we will refer to this  measure as the {\em special flat} measure on $\M^{(1)}$. The following are the main results of this paper. 

\begin{thm}\label{thm: cone measures are horospherical}
The special flat measure on any $G$-orbit closure is horospherical. 
\end{thm}

We will say that a measure $\mu$ is {\em saddle connection free} if $\mu$-a.e.\;surface has no horizontal saddle connections.

\begin{thm}\label{thm: classification} Up to scaling, the only saddle connection free horospherical measure on a $G$-orbit closure is the special flat measure. 
\end{thm}

We emphasize that horospherical measures are a \textit{a priori} not assumed to be finite. It is thus a consequence of Theorem \ref{thm: classification} that horospherical measures are finite under the saddle connection free assumption; it seems likely, but we were not able to prove, that all horospherical ergodic measures are finite. Theorem \ref{thm: classification} was announced without proof in \cite[Claim 1, \S 9]{BSW}.  The saddle connection free assumption cannot be removed; for example, the length measure on a periodic horocycle trajectory in a closed $G$-orbit is horospherical. In \S \ref{sec: examples} we will give more interesting examples of invariant subvarieties and horospherical measures on them, which are not the special flat measure. We will also classify (see \S \ref{subsec: eigenform loci classification}) all the horospherical measures on the simplest nontrivial invariant subvarieties, namely the eigenform loci in $\HH(1,1)$.

If a surface has a horizontal cylinder then so does any surface on its horospherical leaf. We will say that a leaf of the horospherical foliation is {\em cylinder-free} if all surfaces on the leaf have no horizontal cylinders. We say that a measure  $\mu$  on $\M$ is cylinder-free if $\mu$-a.e.\;surface has no horizontal cylinders. In \S \ref{sec: examples} we give examples of horospherical measures which are not special flat  and for which almost every point has a horizontal saddle connection. For these measures it is also the case that almost every point has a cylinder. 
It seems likely that
this is always the case; or in other words, that in Theorem \ref{thm:
  classification} the condition `saddle connection free' can be
weakened to `cylinder-free'. 
The analogous assertion about orbit closures is true:
\begin{thm}\label{thm: densehorosphere}
Any cylinder-free leaf for the horospherical foliation of a $G$-orbit closure is dense in that $G$-orbit closure. 
\end{thm}

The proof of Theorem \ref{thm: densehorosphere} uses a statement of independent interest (Theorem \ref{thm: Apisa Wright}), about extending horizontal saddle connections while staying inside  invariant suborbifolds. This result was explained to us by Paul Apisa and Alex Wright, and its proof is given in Appendix \ref{appendix: extending cylinders}. 



The {\em geodesic flow} is the restriction of the $G$-action to the subgroup
\begin{equation}\label{eq: geodesic subgroup}
A \df \left\{g_t : t \in \R\right\} \subset G, \ \ \text{ where } g_t  \df \left(\begin{matrix} e^t & 0 \\ 0 &
      e^{-t}  \end{matrix} \right).  
\end{equation}
Answering a question of Forni, we prove:

\begin{thm}\label{thm: Forni question}
For any finite horospherical measure $\mu$ on $\M$, the pushforward measures, $g_{t*}\mu$, converge to the special flat measure on $\M$, with respect to the weak-$*$ topology, as $t \to +\infty.$
\end{thm}
Related results are proved in \cite{Forni_density_one}; we stress however that the notion of `horospherical measure' used in \cite{Forni_density_one} is different from the one we use here. 
From a dynamical perspective, the horospherical foliation is the strong unstable foliation for the geodesic flow. Our arguments yield a simpler proof of the following theorem.

\begin{thm}[\cite{EM, EMM}]\label{cor: maximal entropy}
The special flat measure is the unique $A$-invariant horospherical measure on any $G$-orbit closure. Any leaf for the weak-unstable foliation on any $G$-orbit closure is dense.
\end{thm}

\begin{rem}
    Note that we do not assume that the measure is finite in Theorem \ref{cor: maximal entropy}. If we assumed finiteness, then the first statement would follow immediately from Theorem \ref{thm: Forni question}. Also note that Theorem \ref{thm: Forni question} is false for infinite measures, as the following example shows. Take $\M = \GL q \simeq \GL/\Gamma$ a closed orbit of a Veech surface with Veech group $\Gamma$; then a horospherical measure in this case is just a $U$-invariant measure. For $s>0$, let $\nu_s$ be a normalized length measure on a periodic $U$-orbit of length $s$, and let $\mu = \int_{0}^1 a(s) \nu_s \, ds$, where the function $a:[0,1] \to \R_{>0}$ satisfies $\int_0^1 a(s) \, ds = \infty$. Clearly $\mu $ is infinite and Radon. For $\eta>0$, we define $K$ to be the compact subset of $Gq$ consisting of surfaces without saddle connections shorter than $\eta$. Then $g_{t*}\mu(K) \to_{t \to\infty} \infty,$ and hence any weak-* limit of $g_{t*}\mu$ is not Radon. 
\end{rem}



\subsection{Further motivation, prior work, and some ideas from the proofs} 

The work of Eskin, Mirzakhani and Mohammadi gives a very detailed understanding of invariant measures and sets for the $G$-action and the $P$-action on strata of translation surfaces. A central remaining open problem is to understand horocycle invariant ergodic measures. Such an understanding would have an application to the fundamental problem of asymptotic growth of saddle connections on translation surfaces or rational billiards (see \cite{eskin_masur_2001}).
As we will see in \S \ref{sec: horospherical measures}, horospherical measures are horocycle-invariant; thus understanding horospherical measures can be seen as a contribution to the problem of understanding  general horocycle-invariant measures.

A previous measure rigidity result for horospherical measures was obtained in 2008, independently by Lindenstrauss and Mirzakhani \cite{lindenstrauss2008ergodic} and by Hamenst\"adt \cite{hamenstadt2009invariant}. They were interested in understanding mapping class group invariant measures on the space of measured laminations. By a `duality principle' (see \cite[\S 5]{lindenstrauss2008ergodic}) this question is very closely related to the problem of  classifying horospherical measures on the principal stratum. 

%
%
%
%
%
%

Our  argument for Theorem \ref{thm: classification} follows \cite{lindenstrauss2008ergodic}, which in turn is inspired by ideas of Dani \cite{dani1978invariant} and Margulis \cite{Margulis_thesis} The main ingredients are the mixing of the $A$-action, the use of dynamical boxes and how they transform under the $A$-action,  and nondivergence results for the $U$-action (which in the present context were obtained in \cite{MW}). After the requisite preparations, this argument is given in \S \ref{sec: main proof}. In order to carry out the details of this argument, we give a precise description of horospherical measures and special flat measures, and their decomposition into conditional measures in flow boxes in \S \ref{sec: horospherical measures}.
Theorem \ref{thm: densehorosphere} is proved in \S\ref{sec: leaf closures}.
Theorems \ref{thm: Forni question} and \ref{cor: maximal entropy} are proved in \S\ref{sec: geo and fol}.

\subsection{Acknowledgements}
We are grateful to Paul Apisa and Alex Wright for providing the proof
of Theorem \ref{thm: Apisa Wright}. The proof is given in Appendix
\ref{appendix: extending cylinders}. We are also grateful to Giovanni Forni for useful comments. We acknowledge support from grants BSF 2016256,  ISF 2019/19
and ISF-NSFC 3739/21.

\section{Preliminaries}\label{sec: prelims}
In this section we introduce our objects of study and set up our
notation. There are many approaches to these definitions. In our
approach, the linear orbifold structure (or affine orbifold structure) given by period coordinates will be important and we will stress this point of view in what follows. A suitable reference for the theory utilizing this point of view is  \cite[\S2]{BSW}, and unless stated otherwise, our notation, terminology and assumptions are as in \cite{BSW}. 
See also 
\cite{MT, Zorich_survey, Forni-Matheus_survey, Wright_survey}.
See \cite{Goldman} for a general discussion of affine manifolds.

\subsection{Strata and period coordinates}\label{subsec: strata}

Let $S$ be a connected, compact orientable surface of genus $g$, $\Sigma = \{\xi_1, \ldots, \xi_k \} \subset S$ a finite set, $a_1, \ldots, a_k$ non-negative integers with $\sum a_i = 2g-2$, and $\HH = \HH(a_1, \ldots, a_k)$ the corresponding stratum of translation surfaces. We let $\HH_{\mathrm{m}}=\HH_{\mathrm{m}}(a_1, \ldots, a_k)$ denote the stratum of marked translation surfaces and $\pi : \HH_{\mathrm{m}} \to \HH$ the forgetful mapping. It will be useful to assume that singular points are labeled, or equivalently, $\HH = \HH_{\mathrm{m}} / \Mod(S, \Sigma)$, where $\Mod(S, \Sigma)$ is the group of isotopy classes of orientation-preserving homeomorphisms of $S$ fixing $\Sigma$, up to an isotopy fixing $\Sigma$. 
We will typically denote elements of $\HH$ by the letter $q$ when we want to consider them as points of $\HH$, and by the letter $M$ or $M_q$ when we want to consider their underlying topological or geometrical properties as spaces in their own right. Points in $\HH_{\mathrm{m}}$ will be typically denoted by boldface letters such as  $\q$. 

We recall the definition of the map $\dev: \HH_{\mathrm{m}} \to H^1\left(S, \Sigma; \R^2\right)$. 
For an oriented path $\gamma$ in $M_q$ which is either closed or has endpoints at singularities, let $\hol(M_q, \gamma) \df \left( \int_\gamma dx_q, \int_\gamma dy_q \right)$, where $dx_q$ and $dy_q$ are the 1-forms on $M_q$ inherited from the the forms $dx$ and $dy$ on the plane. 
Given $\q \in \HH_{\mathrm{m}}$ represented by $f: S \to M_q$, where $M_q$ is a translation surface, we define $\dev(\q) \df f^*(\hol(M_q, \cdot))$. The map $\dev$ is also known in the literature as the \textit{period map}. There is an open cover $\{\mathcal{U}_\tau\}$ of $\HH_{\mathrm{m}}$, indexed by triangulations $\tau$ of $S$ with triangles whose vertices are in $\Sigma$, such that the restricted maps
$$\varphi_\tau \df \dev|_{\mathcal{U}_\tau}, \ \ \varphi_\tau: \mathcal{U}_{\tau} \to H^1\left(S, \Sigma; \R^2\right)$$ 
are homeomorphisms onto their image. The charts $\varphi_\tau$ give an atlas with affine overlap maps and endow $\mathcal{H}_{\mathrm{m}}$ with a structure of affine manifold. This atlas of charts $\{(\mathcal{U}_\tau, \varphi_\tau)\}$ is known as the {\em period coordinate atlas}.  

The $\Mod(S,\Sigma)$-action on $\HH_{\mathrm{m}}$ is properly discontinuous and affine, and hence $\HH$ inherits the structure of affine orbifold, and the map $\pi: \HH_{\mathrm{m}} \to \HH$ is an orbifold covering map. We can associate to any affine manifold a {\em holonomy cover} and a {\em developing map}. In this case $\HHm$ is a cover with trivial holonomy and $\dev$ plays the role of a developing map of $\HH$ (see \cite{Goldman}). 

The group $\GL$ acts on translation surfaces in $\HH$  and $\HH_{\mathrm{m}}$ by modifying planar charts. It acts on $H^1\left(S, \Sigma; \R^2\right)$ via its action on the coefficients $\R^2$.  The $\GL$-action commutes with the $\Mod(S, \Sigma)$-action, and thus the map $\pi$ is $\GL$-equivariant for these actions. The $\GL$-action on $\HH_{\mathrm{m}}$ is free, since $\dev(g\q) \neq \dev(\q)$ for any nontrivial $g \in \GL$. 

We have a coordinate splitting of $\R^2$ and we write $\R^2 = \R_{\mathrm{x}} \oplus \R_{\mathrm{y}}$ to distinguish the two summands in this splitting. There is a corresponding splitting of cohomology
\begin{equation}\label{eq: splitting horizontal vertical}
  H^1\left(S, \Sigma; \R^2\right) = H^1(S, \Sigma; \R_{\mathrm{x}})  \oplus
  H^1\left(S, \Sigma; \R_{\mathrm{y}}\right). 
\end{equation}
We refer to the summands in this splitting as the {\em horizontal space} and {\em vertical space} respectively.

 It can also be useful to identify the coefficients with $\C$ and consider $H^1(S, \Sigma; \C)$. This is the most natural choice when we are considering Abelian differentials.
An {\em $\R$-structure} on a complex vector space $V$ is given by a choice of a real subspace $W\subset V$ so that $V=W\oplus \mathbf{i}W$. 
If $V$ is equipped with an $\R$-structure we say that a complex subspace $V'\subset V$ is {\em defined over $\R$} if $V'=W'\oplus {\mathbf{i}}W'$ for some real subspace $W'\subset W$.
We give the complex vector space $V=H^1(S, \Sigma; \C)$ the $\R$-structure corresponding to the real subspace $W=H^1(S, \Sigma; \R)=H^1(S, \Sigma; \R_{\mathrm{x}})\subset H^1(S, \Sigma; \C)$. In this language $\mathbf{i} W=\mathbf{i}H^1(S, \Sigma; \R)=H^1(S, \Sigma; \R_{\mathrm{y}})$.

More generally, if $V$ is  a complex vector space with an $\R$-structure, then $\GL$ acts on $V$, with the matrix 
\[
\begin{pmatrix}
    a&b\\
    c&d
\end{pmatrix}
\]
sending $v = w_1 + \mathbf{i} w_2$ to $(a w_1 + b w_2) + \mathbf{i}(c w_1 + d w_2)$. 

\begin{lem}\label{defined over R}
Let $V$ be a complex vector space with an $\R$-structure, and $V'$ be a real subspace. The following are equivalent:
\begin{enumerate}
    \item \label{item: 11}
    $V'$ is invariant under the action of $\GL$.
    \item \label{item: 12}
    $V'\subset H^1(S, \Sigma; \C)$ is a complex subspace defined over $\R$.
\end{enumerate}
\end{lem}

\begin{proof} 
The implication  \eqref{item: 2} $\implies$ \eqref{item: 1} is clear from the definitions. We prove \eqref{item: 1} implies \eqref{item: 2}. If $V'$ is invariant under $\GL$, then since it is a closed subset of $V$, it is mapped into itself by any $2$-by-$2$ matrix, invertible or not.
Let 
\[a\df
\begin{pmatrix}
    0&-1\\
    1&0
\end{pmatrix}, \ \ \ b \df
\begin{pmatrix}
    1&0\\
    0&0
\end{pmatrix}, \ \ \ \text{ and } c \df aba^{-1} = 
\begin{pmatrix}
    0&0\\
    0&1
\end{pmatrix}.
\]
From the definition of the $\GL$ action, one sees that multiplication by $a$ corresponds to multiplication by $\mathbf{i}$, and multiplication by $b$ and $c$ correspond to projections onto 
the two summands in \eqref{eq: splitting horizontal vertical}. Invariance by $a$ implies that $V'$ is a complex subspace, and from the relations $bV' \subset V'$ and $cV' \subset V'$ and $b + c = \mathrm{Id}$, we see that $V'$ is defined over $\R$.
\end{proof}

\begin{rem}
 If $V$ has an $\R$-structure, then so does its dual space, so it makes sense to say that a linear function on $V$ is real. A complex subspace of $V$ is defined over $\R$ if and only it cut out by real linear functions. We will not use this description in this paper.
\end{rem}

We have a restriction map $\Res: H^1(S, \Sigma; \R^2) \to H^1(S; \R^2)$ (given by restricting a cochain to closed paths). Since $\Res$ is topologically defined, its kernel $\ker ( \Res)$ is $\Mod(S, \Sigma)$-invariant. 
Moreover our convention that singular points are marked implies that the $\Mod(S, \Sigma)$-action on $\ker ( \Res)$ is trivial.

Define the \emph{real REL space}
\begin{equation}\label{eq: def Z}
  Z \df\ker ( \Res )\cap H^1(S, \Sigma; \R_{\mathrm{x}}).
\end{equation}
For any $v \in Z$ the constant vector field on $H^1(S, \Sigma; \R^2)$ in direction $v$ pulls back to a well-defined vector field on $\HH_{\mathrm{m}}$ via the local diffeomorphism $\mathrm{dev}$. Since monodromy acts trivially on $Z$, this descends to a vector field on $\HH$. Integrating this vector field gives a locally defined {\em real REL flow (corresponding to $v$)} $(t, q) \mapsto \Rel_{tv}(q)$. For every $q \in \HH$ a trajectory is defined for $t \in I_q$, where the {\em domain of definition} $I_q  = I_q(v) $ is an open interval of $\R$ which contains $0$. This interval is all of $\R$ if the underlying surface $M_q$ has no horizontal saddle connections. If $q \in \HH, \, s \in \R$ and $t \in I_q$ then $t \in  I_{u_sq}$, and $\Rel_{tv}(u_sq) = u_s \Rel_{tv}(q)$. The set
\begin{equation}\label{eq: set where defined}
Z^{(q)} \df  \{v \in Z: \Rel_{v}(q) \text{ is defined} \} = \{v \in Z: 1 \in I_q(v)\},
\end{equation}
as well as the sets $I_q(v)$, are explicitly described in \cite[Thm. 6.1]{BSW}. 

\subsection{Invariant subvarieties}\label{subsec: invariant subvarieties}

In this subsection, we introduce our notion of {\em invariant subvarieties} and {\em irreducible invariant subvarieties}. It will be shown in \cite{SY}, using the work of Eskin-Mirzakhani \cite{EM} and Eskin-Mirzakhani-Mohammadi \cite{EMM}, that an irreducible invariant subvariety is exactly a $\GL$-orbit closure while an invariant subvariety is a finite union of such $\GL$-orbit closures. 

\begin{defin}
 A {\em $d$-dimensional linear manifold} is a submanifold $L$ of $\HHm$ which is a connected component of $\dev^{-1}(V)$ where $V$ is a $d$-dimensional complex subspace of $H^1\left(S,\Sigma;\R^2\right)$ defined over $\R$.
\end{defin}

 Since the developing map is equivariant and $\Mod(S,\Sigma)$ acts linearly on the space $H^1\left(S,\Sigma;\R^2 \right)$, it follows that $\Mod(S,\Sigma)$ takes a $d$-dimensional linear manifold to a $d$-dimensional linear manifold.
If $L$ is a linear manifold corresponding to $V_L \subset H^1\left(S,\Sigma;\R^2 \right)$, we denote by $\Gamma_L$ be the subgroup of $\Mod(S,\Sigma)$ that preserves $L$. Since the developing map $\dev$ is $\Mod(S,\Sigma)$-equivariant, we get an induced action of $\Gamma_L$ on $V_L$. We say that $L$ is an {\em equilinear manifold} if furthermore we have $\det  \left(\gamma|_{V_L} \right)= \pm 1$ for every $\gamma \in \Gamma_L$.  

\begin{defin}\label{defin: invariant subvariety}
A {\em $d$-dimensional invariant subvariety} is a subset  $\M \subset \HH$ such that $\pi^{-1}(\M)$ is a locally finite union of $d$-dimensional equilinear manifolds.
\end{defin}

We will write $d = \dim (\M)$; in some texts this is referred to as the complex dimension of $\M.$
The term ``invariant" in the definition  of invariant subvariety is justified by the following: 

\begin{prop}\label{prop: invariant by the action of GL}
An invariant subvariety is  closed and $\GL$-invariant.
\end{prop}

\begin{proof}
Since $V$ is a closed subset of $H^1\left(S,\Sigma;\R^2\right)$ it follows that  $\dev^{-1}(V)$ is a closed subset of $\HHm$. It follows that a linear manifold is a closed subset of $\HHm$.
The set $\pi^{-1}(\M)$ is closed because it is a locally finite union of closed sets, and this implies that $\M$ is closed.

Since $\pi$ is $\GL$-equivariant, it is enough to prove that $\pi^{-1}(\M)$ is $\GL$-invariant. Let $L$ be a linear submanifold contained in $\pi^{-1}(\M)$ which maps to $V_L$ under $\dev$. By definition, $V_L$ is defined over $\R$ and by Lemma \ref{defined over R} it is invariant under the action of $\GL$ on $H^1\left(S,\Sigma;\R^2 \right)$. Since $\dev$ is $\GL$-equivariant the action of $\GL$ on $\HHm$ preserves $\dev^{-1}(V_L)$. Since $\GL$ is connected, the action of $\GL$ on $\HHm$ preserves $L$. Since $\pi^{-1}(\M)$ is a union of linear sub-manifolds it follows that it is invariant under $\GL$.  
\end{proof}

\begin{defin}
    A $d$-dimensional invariant subvariety is said to be {\em irreducible} if it cannot be written as a union of two proper distinct $d$-dimensional invariant subvarieties. 
\end{defin}

We have the following equivalent characterization:

\begin{prop}\label{characterization of irreducible}
    Let $\M$ be a $d$-dimensional invariant subvariety. Then $\M$ is irreducible if and only if for any d-dimensional equilinear manifold $L \subset \pi^{-1}(\M)$, we have 
     \begin{equation}\label{eq: what we want equilinear}
    \bigcup_{\gamma \in \Mod(S,\Sigma)} L \cdot \gamma = \pi^{-1}(\M).
    \end{equation}
\end{prop}

For the proof of Proposition \ref{characterization of irreducible} we will need the following: 
    \begin{lem}\label{lem: meagre} If $L$ and $L'$ are distinct $d$-dimensional linear submanifolds, then  $\pi(L)\cap \pi(L')$ is a meager subset of $\pi(L)$ and of $\pi(L')$.
 \end{lem}
 \begin{proof}
 We first show that $\pi^{-1}\left(\pi(L)\cap \pi(L') \right)$ is a countable union of sets of dimension less than $d$. We have:
 $$
 \pi^{-1}(\pi(L))=\bigcup_{\gamma \in \Mod(S,\Sigma)} L\cdot \gamma \mathrm{\ \ \ and\ \ \ } 
 \pi^{-1}(\pi(L'))=\bigcup_{\gamma \in \Mod(S,\Sigma)} L'\cdot \gamma.
 $$

Now consider an intersection  $(L\cdot \gamma) \cap (L'\cdot \gamma')$. We have $\dev(L\cdot \gamma)\subset V$ and $\dev(L'\cdot \gamma)\subset V'$ for $d$-dimensional linear subspaces of $H^1(S,\Sigma;\R^2)$. 
 If $V=V'$ then $L\cdot \gamma$ and $ L'\cdot \gamma'$ are topological components of $\dev^{-1}(V)$ so they are either disjoint or equal. If $L\cdot \gamma=L'\cdot \gamma'$ then $L=L'\cdot \gamma'\gamma^{-1}$ so 
 $\pi^{-1}(\pi(L))=\pi^{-1}(\pi(L'))$ and $\pi(L)=\pi(L')$ contrary to assumption.
 
 If $V\ne V'$ then $(L\cdot \gamma) \cap (L'\cdot \gamma')\subset\dev^{-1}(V\cap V')$. This is a complex subspace of positive codimension so its inverse image is a nowhere dense subset of the $d$-dimensional manifolds $L\cdot \gamma$ and $ L'\cdot \gamma'$. 
 Thus $\pi^{-1}\left(\pi(L)\cap\pi(L')\right)$ is a meager subset of $\pi^{-1}(\pi(L))$ and $\pi^{-1}(\pi(L'))$. Since $\pi$ is open and the intersections $(L \cdot \gamma) \cap (L' \cdot \gamma')$ are closed, the projection $\pi(L)\cap \pi(L')$ is a meager subset of $\pi(L)$ and $\pi(L')$.
 \end{proof}
 

\begin{proof}[Proof of Proposition \ref{characterization of irreducible}]
Say that $\M$ is irreducible and let $L$ be a $d$-dimensional equilinear manifold in $\pi^{-1}(\M)$. If \eqref{eq: what we want equilinear} does not hold, we can write $\pi^{-1}(\M)$ as a countable union of orbits of distinct linear submanifolds $L_1,L_2,\ldots$, as 
$$
\pi^{-1}(\M)= \bigcup_{\ell} \bigcup_{\gamma \in \Mod(S,\Sigma)} L_\ell \cdot \gamma, 
 $$
where $L=L_1$ and the list $\{L_i\}$ contains more than one element. We have 
 $$
\M= \bigcup_{\ell} \pi( L_\ell).
 $$
 We define
 $$
 A \df  \pi(L_1) \mathrm{\ and\ } B \df \bigcup_{1<\ell} \pi(L_\ell).
 $$
Since $\M$ is irreducible, and since we have assumed that  \eqref{eq: what we want equilinear} fails, we have $\M=B$. This implies $A\subset B$, and hence  $\pi(L_1)=\bigcup_\ell \pi(L_1)\cap \pi(L_\ell)$.
 According to Lemma~\ref{lem: meagre} $\pi(L_1)\cap \pi(L_\ell)$ is a meager subset of $\pi(L_1)$ so our decomposition of $\pi(L_1)$ expresses $\pi(L_1)$ as a meager set and violates the Baire category theorem. We conclude that $\M=A$ which is what we wanted to show.

 Now assume that for any $d$-dimensional equilinear manifold $L \subset \pi^{-1}(\M)$ we have \eqref{eq: what we want equilinear}. Suppose we have a decomposition $\M=A\cup B$ where
$$
 A= \bigcup_{j} \pi(L_j) \mathrm{\ and\ } B=\bigcup_{k} \pi(L'_k),
 $$
  where both collections $\{L_i\}, \, \{L'_k\}$ are $\Mod(S, \Sigma)$-invariant and comprised of $d$-dimensional equilinear manifolds. By \eqref{eq: what we want equilinear}, $\pi^{-1}(A)$ and $\pi^{-1}(B)$ are either empty or equal to $\pi^{-1}(\M)$. Thus $A$ and $B$ are not proper subsets of $\M$. 
\end{proof}

It follows from Proposition \ref{characterization of irreducible} that if $\M$ is a $d$-dimensional irreducible invariant subvariety and $L$ is a $d$-dimensional equilinear manifold contained in $\pi^{-1}(\M)$, then $\pi(L)=\M$. This motivates the following definition that will be used throughout the text 

\begin{defin}
    Let $\M$ be a $d$-dimensional irreducible invariant subvariety. A {\em lift} of $\M$ is a $d$-dimensional equilinear manifold $L \subset \pi^{-1}(\M)$. 
\end{defin}

The following result 
establishes the link between $\GL$-orbit closures and invariant subvarieties. In the forthcoming 
\cite{SY}, it will be deduced from the results of \cite{EMM, EM}.  

\begin{thm}
Irreducible invariant subvarieties and $\GL$-orbit closures coincide. Furthermore, any invariant subvariety is a finite union of irreducible invariant subvarieties. 
\end{thm}

\begin{conv}
From now on we will make the standing assumption that all the invariant subvarieties we will consider are irreducible. 
\end{conv}

Let $\M$ be a $d$-dimensional invariant subvariety. We conclude this section by constructing a Radon measure supported on $\M$ which will be defined up to a multiplicative constant. This will require some constructions which are summarized in Appendix \ref{appendix: correspondence principle}. Let $L$ be a lift of $\M$, let $V_L = \dev(L)$ and let $\Gamma_L$ be the stabilizer in $\Mod(S,\Sigma)$ of $L$. Let $\alpha$ be a volume form on $L$ that is obtained as the pullback by $\dev$ of an element of the top degree exterior power of $V_L$. The group $\GL$ acts smoothly on $\HHm$. Denoting by $g^\ast$ the pull-back operator on differential forms corresponding to the action of $g \in \GL$ on $\HHm$, we have 
\begin{align}\label{eq: invariance property}
\forall g \in \GL, \ \ g^{\ast} \alpha = (\mathrm{det} \ g)^{d} \ \alpha.
\end{align}

\noindent The volume form $\alpha$ defines a measure on $L$ that we denote by $\mu_L$. Since $L$ is an equilinear manifold, the measure $\mu_L$ is $\Gamma_L$-invariant. Furthermore, since $\Mod(S,\Sigma)$ acts transitively on the set of irreducible components, it can be arranged that for any $\gamma \in \Mod(S,\Sigma)$, $\gamma_{\ast} m_{L} = m_{L \cdot \gamma}$. This means that the sum
$$ 
\tilde{\mu}_{\M} = \sum_{L \subset \pi^{-1}(\M)} \mu_{L}, 
$$ 
\noindent where the sum ranges over the lifts of $\M$,
is a $\Mod(S,\Sigma)$-invariant measure on $\HHm$.
The measure $\tilde{\mu}_{\M}$ is a Radon measure, which follows from the fact that the collection of irreducible components is locally finite. Using Proposition \ref{prop: correspondence}, there is a unique Radon measure $\mu_{\M}$ on $\HHm$ such that for any $f \in C_c(\HHm)$, we have
\begin{align*}
    \int_{\HHm} f \ d\tilde{\mu}_{\M} = \int_{\HH} \Big( \int_{\HHm} f \ d\theta_q \Big) \ d\mu_{\M}(q),
\end{align*}
where 
$$
\theta_q \df \sum_{\bq \in \pi^{-1}( q)} N(\bq) \cdot \delta_{\bq}, \ \ N(\bq) \df  |\{\gamma \in \Mod(S, \Sigma) : \bq = \bq \cdot \gamma\}|
$$
(as in equation~\eqref{eq: appendix counting measure}). 
The measure $\mu_{\M}$ is supported on $\M$. It follows from Lemma~\ref{lem:invariance} that it is $\SL$-invariant. We call it the {\em linear measure on $\M$}. Notice that this is a slight abuse of language as $\mu_{\M}$ is only determined up to a multiplicative constant.  
\subsection{Area one locus, cone construction, and special linear measures} Let $\q \in \HH_{\mathrm{m}}$, let $q =\pi(\q)$, and let $M= M_q$ be the  underlying translation surface. The area of $M$ can be expressed using period coordinates as follows.
We define a Hermitian form on $H^1(S, \Sigma ; \C)$ by 
\begin{equation}\label{eq: Hermitian form}
( \alpha, \beta ) = \frac{1}{2 \mathbf{i}} \int_{S} \alpha \wedge \bar \beta.
\end{equation}
(See \cite[\S 2.5]{BSW} for a topological interpretation of equation~\eqref{eq: Hermitian form}.) The area of $M$ is then given by $(\dev(\q), \dev(\q))$. This is thus a quadratic formula in period coordinates. 

For the purposes of this paper we will use a related {\em real valued} bracket $\langle \alpha, \beta \rangle $ involving the pairing of horizontal and vertical classes. Say that on a marked surface $M_q$ we have a 1-form $\alpha$ corresponding to an element of $H^1(S, \Sigma; \R_{\mathrm{x}})$ (a horizontal form) and a 1-form $\beta$ corresponding to an element of $H^1(S, \Sigma; \R_{\mathrm{y}})$ (a vertical form). Then 
\begin{equation*}
\langle \alpha, \beta \rangle =  \int_{S} \alpha \wedge \beta
\end{equation*}
and this gives
\begin{equation}\label{eq: area formula bracket}
\area(M_q)=\langle dx_q, dy_q \rangle =  \int_{S} dx_q \wedge dy_q .
\end{equation}

We denote the subset of surfaces in $\HHm$ and $\HH$ of area one by $\HHm^{(1)}$ and $\HH^{(1)}$. More generally, when $\M$ is an invariant subvariety and $L$ is a lift of $\M$, we also denote by $\M^{(1)}$ and $L^{(1)}$ their intersection with the area-one locus. The latter are $G$-invariant and invariant under real REL flows (where defined). 

We recall that there is a {\em rescaling action} of $\R_+^{\ast}$ on $\HH$ that corresponds to the action of the subgroup of $\GL$ of scalar matrices with positive coefficients. We consider the {\em cone measure} $m_{\M}$ on $\M^{(1)}$ defined for any Borel subset $A \subset \M^{(1)}$ by 
\begin{align}\label{eq: special flat measure}
    \cone(A) \df \mu_{\M}(\mathrm{cone}(A)), \ \ \text{where} \ \mathrm{cone}(A) \df \{t\cdot a : t \in (0,1], \ a \in A\}.
\end{align}
When $\M$ is the whole stratum $\HH$, the measure $m_{\HH}$ is called the {\em Masur-Veech measure}. More generally, we shall call the measure $\cone$ the {\em special flat measure} on $\M$. If $L$ is a lift of $\M$, we can perform the same cone construction with the measure $\mu_{L}$ and we denote by $m_{L}$ the corresponding measure. 
Let $\tilde{m}_{\M}$
be the pre-image of $m_{\M}$ under $\pi$, that is the unique measure 
on $\HHm$ such that 
for any $f \in C_c(\HHm)$,
\begin{align}\label{eq: lifted measure}
        \int_{\pi^{-1}(\M)} f \ d\tilde{m}_{\M} = \int_{\M} \Big(\int_{\tilde X} f \ d\theta_q \Big) \ dm_{\M}(q)
    \end{align}
    (see 
Definition \ref{def: preimage of measure}). It is easily verified that 
\begin{align}\label{irreducible components of the cone measure}
    \tilde{m}_{\M} = \sum_{L \subset \pi^{-1}(\M)} m_{L}.
\end{align}

\subsection{The sup-norm Finsler metric} \label{subsec: AGY metric}
We now recall the sup-norm  Finsler metric on $\HHm$.  This structure was studied by Avila, Gou\"ezel and Yoccoz, for proofs and more details see \cite{AGY} and \cite{avila2010small}. Let $\|\cdot \|$ denote the Euclidean norm on $\R^2$. For a translation surface $q$, denote by $\Lambda_q$ the collection of saddle connections on $M_q$ and let $\ell_q(\sigma)= \|\hol_q(\sigma)\|$ be the length of $\sigma \in \Lambda_q$. For $\beta\in H^1(M_q,\Sigma_q; \R^2)$ we set 
\begin{equation}\label{eq: define a norm downstairs}
\|\beta\|_q \df \sup_{\sigma \in \Lambda_q}
\frac{\|\beta(\sigma)\|}{\ell_q(\sigma)}.
\end{equation}

We now define a Finsler metric for $\HHm$. Let $f: S \to M_q$ be a marking map representing a marked surface $\q \in \HHm$. Using period coordinates we can  identify the tangent space to $\HHm$ at $\q$ with $H^1(S,\Sigma;\R^2)$.  
Then 
\begin{equation}\label{eq: define a norm upstairs}
\|\beta\|_{\q} \df \sup_{\tau \in \Lambda_{\q}} \frac{\|\beta(f(\tau))\|}{\ell_q(f(\tau))}
\end{equation}
is a norm on $H^1(S, \Sigma ; \R^2)$. It satisfies the equivariance property 
\begin{equation}\label{eq: equivariance}
\forall h \in \Mod(S,\Sigma), \ \ \|\beta\|_{\q} = \|h^*\beta\|_{\q \cdot h}, 
\end{equation}
where $\q \cdot h $ is represented by the marking map $f \circ h$. 
The map 
$$
T(\HH_{\mathrm{m}}) \to \R, \ \ \ \ \ \ \ (\q,\beta) \mapsto
\|\beta\|_{\q}
$$
is continuous.  
The Finsler metric  defines a distance function\footnote{In order to avoid confusion we use `distance function' to refer to what is often called a metric.} on $\HHm$ which we call the {\em sup-norm distance} and define as follows:
\begin{equation}\label{eq: Finsler integrate}
    \mathrm{dist}(\q_0, \q_1) \df \inf_{\gamma } \int_0^1 \|\gamma'(\tau)\|_{\gamma(\tau)} d\tau,
\end{equation}
where $\gamma$ ranges over smooth paths $\gamma:[0,1] \to \HH$ with $\gamma(0)=\q_0$ and $\gamma(1) = \q_1$.  The topology induced by the sup norm distance on  $\HHm$ is the one  induced by period coordinates, and the resulting metric space  is proper and complete. We can use the distance function on $\HHm$ to define a distance function on $\HH$ by
$$
\mathrm{dist}(q_0, q_1) = \inf \{\mathrm{dist}(\q_0, \q_1): \q_i \in
\pi^{-1}(q_i), \ i=0,1\}.
$$

\section{Horospherical measures}\label{sec: horospherical measures}

Let $\M$ be an invariant subvariety of dimension $n$. The goal of this section is to define the horospherical foliation on $\M$ and the related horospherical measures, which are our object of study in this paper. These objects will be defined via their counterparts for the irreducible components of $\pi^{-1}(\M)$.

\subsection{Boxes}\label{subsec: boxes} We now define a notion of {\em boxes}. They will be used throughout the text and will play two roles: boxes give local coordinates on invariant subvarieties (more precisely, on the irreducible components of their pre-image by $\pi$) that are convenient for the study of horospherical measures; additionally, they will be used in a mixing argument in the proof of Theorem \ref{thm: classification}.

From now on, we identify $H^{1}(S,\Sigma,\C)$ with $H^{1}(S,\Sigma,\R^2)$ as in \S \ref{subsec: strata}.  Let $V \subset H^{1}(S,\Sigma,\C) $ be a complex linear subspace defined over $\R$. We have 
\begin{equation}\label{eq: splitting}
    V = V_{\mathrm{x}} \oplus V_{\mathrm{y}},
  \end{equation}
where
$$
V_{\mathrm{x}} \df V \cap H^1(S, \Sigma; \R_{\mathrm{x}}) \ \  \text{ and}  \ \ V_{\mathrm{y}}   \df V \cap H^1(S, \Sigma; \R_{\mathrm{y}})
$$
are identified by the isomorphism $ V_{\mathrm{x}} \ni v \mapsto \mathbf{i} v \in V_{\mathrm{y}}.$ We define
$$
V^{(1)} \df \{(x,y) \in V: x \in V_{\mathrm{x}}, \  y \in V_{\mathrm{y}}, \ \langle x, y \rangle =1 \}, 
$$
and denote by
\begin{equation}\label{eq: def projections}
  \pi_{\x} : V \to V_{\x}, \ \ \ \ \  \pi_{\y}: V \to V_{\y}
  \end{equation}
the projections corresponding to the direct sum
decomposition \eqref{eq:  splitting}, and by $\pi'_{\x}$ the
projection from $\pi_{\x}^{-1} (V_{\x} \sm \{0\})$ to the
projective space $\mathbf{P}(V_{\x})$. Finally let
\begin{equation}\label{eq: def Psi}
  \Psi:V^{(1)}\to \mathbf{P}(V_{\x})\times V_{\y}, \ \ \ \ \Psi(q)  =
  (\pi'_{\x}(q),\pi_{\y}(q)).
  \end{equation}
\begin{lem} \label{lem: local diffeo}
  The map $\Psi$ is a local diffeomorphism.
\end{lem}

\begin{proof}
Say that $(x_0, y_0) \in V^{(1)}$ is mapped by $\Psi$ to $(\bar{x}_0,y_0)$ in $\mathbf{P}(V_{\x})\times V_{\y}$. We will construct a local inverse. Since $\langle x_0,y_0\rangle=1$ we can find neighborhoods $U_{\x}$ of $x_0$ in $V_{\x}$ and  $U_{\y}$ of $y_0$ in $V_{\y}$ so that $\langle x,y\rangle>0$ for $x\in U_{\x}$ and $y\in U_{\y}$. We define maps
\begin{equation}\label{eq: local inverse}
  \tilde{\psi}:U_{\x} \times U_{\y}\to V^{(1)}, \ \ \ \
  \tilde{\psi}(x,y)=\left(\frac{x}{\langle x,y\rangle},y \right)
\end{equation}
and
\begin{equation}\label{eq: local inverse1}
  \psi:U'_{\x} \times U_{\y}\to V^{(1)}, \ \ \ \
  \psi([x],y) \df \tilde{\psi}(x,y), \ \ \ \text{ where } U'_{\x} \df
  \pi'_{\x}(U_{\x}).
\end{equation}
The map $\tilde{\psi}$ is smooth and descends in a well-defined way to define $\psi$. We see that $\Psi \circ \psi$ is the identity map, i.e., $\psi$ is a local inverse of $\Psi$. 
\end{proof}

\begin{defin}[Boxes]\label{def: box}
Let $L$ be a lift of $\M$ and let $V = \dev(L)$. A {\em box in $L$} is a relatively compact subset $\BB \subset L^{(1)}$ together with a diffeomorphism $\varphi : \product \to \BB$ such that, in the notations above,
\begin{itemize}
  \item $U'_{\x}$ and $U_{\y}$ are open sets in $\mathbf{P}(V_{\x})$ and $V_{\y}$ respectively.
  \item $\Psi \circ \dev \circ \varphi = \mathrm{Id}$.
\end{itemize}
For $y  \in U_{\y}$, the {\em plaque of $y$ in $\BB$} is the set $\LL_y \df \varphi(U'_{\x}\times \{y\})$. 
\end{defin}

The composition in the second item in Definition \ref{def: box} makes sense since $\dev(L^{(1)}) \subset V^{(1)}$, in light of 
equation~\eqref{eq: area formula bracket}.
It should be understood as a choice of a suitable parameterization for boxes. Note that the data $\varphi, \product$ are implicit in the notion of a box, but in order to avoid excessive notation we simply write $\BB$.

More generally, a {\em box in $\pi^{-1}(\M)$} is a box in one of the irreducible components of $\pi^{-1}(\M)$. Such a box $\BB$ will be called {\em regular} if for any $\gamma \in \Mod(S,\Sigma)$ either $\BB \cdot \gamma \cap \BB = \emptyset$ or $\gamma \in \Gamma$, where $\Gamma$ is the stabilizer in $\Mod(S,\Sigma)$ of $\BB$ (\textit{i.e.} the set of $\gamma \in \Mod(S,\Sigma)$ such that $\BB \cdot \gamma = \BB$). When $\BB$ is regular, the map $\pi$ induces a homeomorphism $\BB / \Gamma \to \pi(\BB)$. In particular the image of a regular box by $\pi$ is an open subset of $\M$. Since $\Mod(S,\Sigma)$ acts diagonally on $\mathbf{P}(H^1(S,\Sigma,\R_\x)) \times H^1(S,\Sigma,\R_\y)$, the set of boxes is preserved by the action of $\Mod(S,\Sigma)$. Furthermore, a finite intersection of boxes is a box. Thus, by Lemma \ref{lem: local diffeo}, for every $\q \in \pi^{-1}(\M)$, there is a regular box in $\pi^{-1}(\M)$ containing $\q$. 


\begin{rem}\label{remark: asymmetry}
There is an asymmetry in the definition of a box; we could equally well define a box using $V_{\x}$ and $\mathbf{P}(V_{\y})$, but we will make no use of that kind of box.
\end{rem}

\subsection{Definition of the horospherical foliation}\label{subsec:
  horospherical foliation}
Recall that a smooth map of manifolds is a submersion if its
derivative is of full rank at every point. The implicit function
theorem implies that the connected components of the fibers of a
submersion are the leaves of a foliation. 

\begin{defin} \label{def: foliations upstair}
Let $L$ be a lift of $\M$ and let $V$ be the linear space on which $L$ is modeled. The foliations on $L^{(1)}$ induced by the submersions
$$
\pi'_{\x}  \circ \dev : L^{(1)} \to \mathbf{P}(V_{\x}) \ \ \text{and} \ \ \pi_{\y}\circ \dev : L^{(1)} \to V_{\y}, 
$$
\noindent are called the {\em weak stable} and {\em strong unstable} foliations. They are denoted respectively by $\WW^s$ and  $\WW^{uu}$. The leaf of $\q \in L^{(1)}$ for the weak stable foliation is denoted by $\WW^s(\q)$ and the leaf of $\q$ for the strong unstable foliation is denoted by $\WW^{uu}(\q)$. 
\end{defin}

It follows from Lemma \ref{lem: local diffeo} that these foliations are well-defined, and the leaves of these foliations are everywhere transverse.

\ignore{
\textcolor{red}{There was an idea to discuss this in the language of 
  orbifold foliations and $(G,X)$-structures. 
  I guess Peter's contribution in this direction is in the file `$(G, X)$ structures on 
  orbifolds'... but I didn't understand whether Peter wanted us to use 
  that, and if so, how... I used the stuff in \S 2.3 to remove the 
  need to work with orbifolds. Note that the boxes are only boxes in the 
  nonsingular part of $\M$ and similarly for the foliation. Please check}

Recall that  for $\q \in
  \HH_{\mathrm{m}}$, the {\em local group $\Gamma_{\q}$}
  for the orbifold structure on $\HH$ is the subgroup of $\Mod(S,
  \Sigma)$ fixing $\q$. We will say that a foliation on
  $\HH_{\mathrm{m}}$ induces an {\em orbifold foliation} on $\HH$ if
  the monodromy action permutes the leaves,  and each local group  $\Gamma_{\q}
$ preserves the 
  tangent space to the leaf of $\q$; the leaves of an orbifold
  foliation are thus orbifolds on their own right. 
}

\begin{lem} \label{lem: orbifold foliation}
The action of $\Mod(S, \Sigma)$ permutes the leaves of $\WW^{uu}$. For any leaf $F$, the restriction $\dev|_F$ is a local homeomorphism to an affine subspace of $V$ and with respect to this affine structure, the subgroup $\Gamma_L \df \{\gamma \in \Mod(S, \Sigma) : L \cdot \gamma = L\}$ acts on the leaves of $W_L^{uu}$ by affine maps. 
\end{lem}

\begin{proof}
The monodromy preserves the product splitting $V=V_{\x}\oplus V_{\y}$ and acts linearly on each factor. Thus the monodromy acts projectively on $\mathbf{P}(V_{\x})$. Since $\dev$ is monodromy equivariant, the leaves of the foliations $\WW^{s}$ and $\WW^{uu}$ are permuted by the action of $\Mod(S, \Sigma)$. 

For the second assertion, it is clear from the definitions that $\dev$ maps the leaf $F$ to a set of the form $\{ (x,y_0) \in V : x \in V_\x \ \text{and} \ \langle x, y_0 \rangle =1\} $ for some fixed $y_0 \in V_{\mathrm{y}}$, and by Lemma \ref{lem: local diffeo}, the map $\dev|_F$ is a local diffeomorphism. The last assertion follows from the $\Mod(S, \Sigma)$-equivariance of $\dev$ and the fact that $\Mod(S, \Sigma)$ preserves the bracket $\langle \cdot, \cdot \rangle.$ 
\end{proof}

\begin{rem}
Lemma \ref{lem: orbifold foliation} equips the leaves of the foliation $W^{uu}$ with an affine manifold structure. This structure need  not be geodesically complete. Using real Rel deformations, one easily constructs affine geodesics in a leaf $\WW^{uu}(\q)$ which contain a surface with a horizontal saddle connection whose length  goes to zero as one moves along the leaf. There are additional sources of non-completeness involving surfaces whose horizontal foliation is minimal but not uniquely ergodic, see \cite{MW2}. Furthermore, using \cite[Thm. 1.2]{MW2}, one can show that each leaf $\WW^{uu}(\q)$ is mapped by the developing map homeomorphically to an explicitly described convex domain in $H^1(S, \Sigma; \R_{\x})$, defined by finitely many linear inequalities and equalities. 
\end{rem}

It follows from Lemma \ref{lem: orbifold foliation} that the partition of $L^{(1)}$ given by the leaves $\WW^{uu}$ induces a partition of $\M^{(1)}$. We denote it by $W^{uu}$ and if $q \in \M$, we denote by $W^{uu}(q)$ the element of the partition  that contains $q$. We emphasize that $W^{uu}$ does not depend on the choice of a particular irreducible component used to define it. This is a consequence of the fact that $\Mod(S,\Sigma)$ acts on $H^1(S,\Sigma,\C)$ by real endomorphisms and thus preserves the splitting into real and imaginary parts of cohomology classes.  

\begin{defin}
    A {\em horosphere} is an element of the partition of $W^{uu}$. 
\end{defin}

\begin{rem}
Occasionally, we may call the partition $W^{uu}$ the {\em horospherical foliation of $\M$}, even though $\M$ is generally not a manifold. Even if this will play no role in the rest of the paper, we justify this choice of terminology for the sake of completeness: the invariant subvariety $\M$ can be seen to have the structure of a properly immersed manifold $\M$, i.e., is the image of a manifold $\N$ under a proper orbifold immersion $f : \N \to \HH$ and there is a foliation on $\N$ whose leaves are sent to horospheres by $f$. We can choose $\N$ to be the quotient of $L$ by a finite-index torsion-free normal subgroup $\Gamma_0 $ of $\Mod(S, \Sigma)$ and $f : L / \Gamma_0 \to \HH, \ \q\Gamma_0 \mapsto \pi(\q)$. By Lemma \ref{lem: orbifold foliation}, the horospherical foliation on $L$ descends to a foliation on the manifold $L /\Gamma_0$. The leaves of this foliation are indeed mapped to horospheres and $f$ is an orbifold immersion. The fact that it is proper follows from the fact that the collection of irreducible components of $\pi^{-1}(\M)$ is locally finite.
\end{rem}

Reversing the roles of $\pi_{\x}$ and $\pi_{\y}$, and defining $\pi'_{\y}$ in an analogous fashion, we also define the {\em strong stable} and {\em unstable} foliations $\WW^{ss}$ and $\WW^{u}$ as those induced by the submersion $\pi_{\x} \circ \dev, \, \pi'_{\y} \circ \dev$ respectively. Lemma \ref{lem: orbifold foliation} holds for these foliations as well, with obvious modifications. Summarizing: for every $\q \in L$ we have
$$
\WW^{ss}(\q) \subset \WW^{s}(\q), \ \ \ \WW^{uu}(\q) \subset \WW^{u}(\q),
$$
the leaves $\WW^{ss}(\q)$ and $\WW^{uu}(\q)$ have a natural affine structure and, for $n = \dim (\M)$, we have 
$$
\dim \WW^{ss}(\q) = \dim \WW^{uu}(\q) = n-1, \ \ \ \ \ \dim \WW^{s}(\q) = \dim
\WW^{u}(\q) = n. 
$$

As we saw in \S \ref{subsec: AGY metric}, the sup-norm Finsler metric induces a distance function on $\HH_{\mathrm{m}}$ as a path metric. We will induce distance functions on leaves of the stable and strong stable foliations using the same approach. For $\q_0, \q_1 \in \HH_{\mathrm{m}}$ belonging to the same stable (respectively, strong stable) leaf, we  define $\mathrm{dist}^{(s)}(\q_0, \q_1)$ (respectively, $\mathrm{dist}^{(ss)}(\q_0, \q_1)$) by the formula in equation~\eqref{eq: Finsler integrate}, but making the additional requirement that the entire path $\gamma$ is contained in the stable (respectively strong stable) leaf of the $\q_i$. 

We similarly define $\mathrm{dist}^{(s)}(q_0, q_1)$ and $\mathrm{dist}^{(ss)}(q_0, q_1)$ for $q_0, q_1 \in \HH$ belonging to the same stable (respectively, strong stable) leaf. We will call the distance functions $\mathrm{dist}^{(s)}, \, \mathrm{dist}^{(ss)}$ the {\em stable (resp. strong stable) sup-norm distance function.} 

These distance functions have the following properties:

\begin{prop}\label{prop: properties of restricted AGY metric}
Let $L$ a lift of $\M$ and let $\q_0, \q_1 \in L$.
  
\begin{enumerate}
    \item \label{item: 1} If $\q_0, \q_1$ are in the same stable (resp., strong stable leaf) leaf then $\mathrm{dist}(\q_0, \q_1) \leq \mathrm{dist}^{(s)} (\q_0, \q_1)$ (resp., $\mathrm{dist}(\q_0, \q_1) \leq \mathrm{dist}^{(ss)} (\q_0, \q_1)$).
    \item \label{item: 2} If $\q_0, \q_1$ are in the same strong stable leaf then for all $t \geq 0$, 
    
    $$
    \mathrm{dist}^{(ss)} (g_t \q_0, g_t \q_1) \leq \mathrm{dist}^{(ss)} (\q_0, \q_1).
    $$

    And the same holds for the strong unstable leaf. 
    
    \item \label{item: 2.5} If $\q_1 = g_t \q_0$ for some $t \in \R$ then $\mathrm{dist}^{(s)}(\q_0, \q_1) \leq |t|$.
    \item Statements ~\eqref{item: 1}, ~\eqref{item: 2} and ~\eqref{item: 2.5} also hold in $\HH$, for $q_0, q_1$ in place of $\q_0, \q_1$. 
\end{enumerate} 
  \end{prop} 

\begin{proof}
Assertion \eqref{item: 1} is obvious from definitions, and assertions \eqref{item: 2} and \eqref{item: 2.5} are proved in \cite[\S 5]{avila2010small} (where what we call the strong stable foliation is referred to as the stable foliation). The assertions for $\HH$ follows from the corresponding ones for $\HH_{\mathrm{m}}$. 
\end{proof}

\begin{rem}
Almost everywhere, the horospheres $W^{uu}(q)$ and $W^{ss}(q)$ are actually the unstable and stable manifolds of the geodesic flow. That is, for any $q$, and almost every (with respect to the measure class induced by the affine structure on leaves) $q_1 \in W^{uu}(q), \, q_2 \in W^{ss}(q),$ we have 
$$
    \mathrm{dist} (g_t q, g_t q_1) \underset{t \to -\infty}{\longrightarrow} 0 \ \ \text{and} \ \ \mathrm{dist} (g_t q, g_t q_2) \underset{t \to \infty}{\longrightarrow} 0. 
    $$
This is proved in \cite{veech1986the} (see also \cite{Forni-Matheus_survey}) for $\M = \HH^{(1)}$. The same result for general invariant subvarieties can be proved by adapting the  arguments  used in \cite{Forni-Matheus_survey}. 
\end{rem}

\subsection{Definition of horospherical measures}\label{subsec:
  horospherical measures}

Let $L$ be a lift of $\M$ as in Subsection \ref{subsec: invariant subvarieties} and let $V \subset H^1(S,\Sigma,\C)$ be the subspace on which $L$ is modeled. We write $V = V_\x \oplus V_\y$ as in equation~\eqref{eq: splitting horizontal vertical}. Let $\eta_{\x}$ and $\eta_{\y}$ be the translation invariant volume forms on $V_{\x}$ and $V_{\y}$ determined by a choice of an element of the top degree wedge power of $V_{\x}$ and $V_{\y}$. Define
\begin{align}\label{eq: two forms}
    \alpha_{\x} \df (\pi_{\x} \circ \dev)^*(\eta_{\x}), \ \ \ \ \ \ \alpha_{\y} \df (\pi_{\y} \circ \dev)^*(\eta_{\y}).
\end{align}

We recall that the measure $\mu_{L}$ on $L$ was defined in Section \ref{subsec: invariant subvarieties} as the integral of a volume form $\alpha$. From now on, this form will be chosen so that $\alpha = \alpha_\x \wedge \alpha_{\mathrm{y}}$. We define the {\em Euler vector field} $E$ on $\HHm$ such that for any $\q \in \HHm,$
\begin{align}\label{eq: euler}
E(q)= E_x(\q) \df \frac{\partial}{\partial t}\Big|_{t=0} \begin{pmatrix} e^t & 0 \\ 0 & e^t 
    \end{pmatrix} \cdot \q.
\end{align}

This vector field can be thought of as the tangent vector to the rescaling action, which justifies our choice of terminology. Notice furthermore that the image of $E$ by $\dev$ is the usual Euler vector field $e(v) = v$ on $H^1(S,\Sigma;\C)$. This is due to the fact that $\dev$ is $\GL$-equivariant. Since $L$ is a linear manifold, the vector field $E$ is tangent to it.  We use this to define the form
$$
\xform \df\iota_E \alpha_{\x},
$$
i.e., the contraction of $\alpha_{\x}$ by the Euler field $E$. The restriction of $\xform$ to the leaves of $\WW^{uu}$ induces a volume form. We denote by $\nu_{\beta_{\x}}$ the induced measures. 
We emphasize that this defines a system of measures, one on each leaf $W^{uu}_L(\q)$,  so one should write $\nu_{\xform,\q}^L$ instead of $\nu_{\beta_{\x}}$; we omit this in our notation. We say that a measure $m$ on $L$ is {\em horospherical} if it is supported on $L^{(1)}$ and its conditional measures on the leaves of $\WW^{uu}$ are given by the measures $\nu_{\beta_{\x}}$. More precisely, this means that for any box $\BB$ in $L$, there is a measure $\lambda$ on $U_{\y}$ such that for any compactly supported continuous function $f:L \to \R$, 
\begin{equation}\label{integral upstair}
    \int_{\BB} f \ d\nu = \int_{U_{\y}} \left( \int_{\boldsymbol{L}_y} f \, d\nu_{\beta_{\x}} \right) \ d\lambda(y). 
\end{equation}
     \begin{rem}
The measure $\lambda$ is a so-called `transverse measure' for the horospherical foliation. This means it is a system of measures on sets tranverse to the foliation which is invariant under holonomy along leaves, see \cite[Vol. 1, 10.1.13 \& 11.5.2]{Candel_Conlon}. According to the theory of transverse measures  equation~\eqref{integral upstair} yields as a bijection between horospherical measures and transverse measures. We will not be using this point of view in this paper.

         \end{rem}
    
\begin{rem}
Let $\phi_t$ be a smooth flow acting on $L$. A measure $\nu$ is said to be invariant if for any $t \in \R$ we have $(\phi_{t})_* \nu = \nu$. This definition is equivalent to requiring that the conditional measures of $\nu$ on the orbits of $\phi_t$ be multiples of the Lebesgue measure $dt$, i.e., invariant under  the maps $\phi_s x \mapsto \phi_{t+s}x$ for any fixed $t$. The equivalence can be shown by disintegrating $\nu$ on flow boxes, i.e., boxes whose horizontal plaques are pieces of $\phi_t$-orbits. By Lemma \ref{lem: orbifold foliation}, leaves of $\WW^{uu}$ are modeled on linear subspaces, and thus one could try and define horospherical measures as those that are invariant under translation along the leaves. However, these translations are not part of a globally defined group action; for instance trajectories might escape to infinity in finite time. Our definition of horospherical measures is inspired by the second characterization of invariant measures, where the foliation by orbits of $\phi_t$ is replaced by the strong unstable foliation and the translation invariant measure $dt$ is replaced by $\xmeasure$. 
\end{rem} 

In order to define a notion of horospherical measures on $\M$, we first need some terminology: let $\nu$ be a Radon measure on $\M$ and let $\tilde{\nu}$ be its pre-image by $\pi$ as in equation~\eqref{eq: lifted measure} (see also Appendix \ref{appendix: correspondence principle}). By construction, the measure $\tilde{\nu}$ is supported on $\pi^{-1}(\M)$. If $L$ is a lift of $\M$, then the restriction of the measure $\tilde{\nu}$ to $L$ is called \textit{the lift of $\nu$} corresponding to $L$. More generally, a \textit{lift} of $\nu$ is a measure of the form $\tilde{\nu}|_{L}$ where $L$ is any lift of $\M$.  For instance, the measures $m_{L}$ in equation~\eqref{irreducible components of the cone measure} are the lifts of $m_{\M}$. 

\begin{defin}[Horospherical measure] A Radon measure $\nu$ on $\M$ is {\em horospherical} if its lifts are horospherical. 
\end{defin} 

By Proposition \ref{characterization of irreducible}, it is enough that one of the lifts is horospherical, as the action of $\Mod(S,\Sigma)$ preserves the set of horospherical measures on $\HH_{\mathrm{m}}$. By definition, a horospherical measure on $\M$ is supported on $\M^{(1)}$. We have the following useful local disintegration formula: 
\begin{prop}
    Let $\nu$ be a horospherical measure on $\M$. For any regular box
    $\varphi : \product \to \BB$
    in $\pi^{-1}(\M)$, there is a measure $\lambda$ on $U_{\y}$ such that for any compactly supported continuous function $f:\M \to \R$, denoting $B=\pi(\BB)$ we have
\begin{equation}\label{integral}
    \int_{\B} f \ d\nu = \int_{U_{\y}} \left( \int_{\boldsymbol{L}_y} f \circ \pi \ d\nu_{\xform} \right) \ d\lambda(y). 
\end{equation}

\end{prop}

\begin{proof}
Let $L$ be a lift of $\M$ in which $\BB$ is contained. We denote by $\Gamma$ the stabilizer in $\Mod(S,\Sigma)$ of $\BB$, and by $\tilde{\nu}$  the pre-image of $\nu$ under $\pi$. By definition, the measure $\tilde{\nu}|_{L}$ is horospherical,  and we let $\lambda_0$ be a measure on $\U_{\y}$ as in equation~\eqref{integral upstair}. We set $\lambda \df \frac{1}{|\Gamma|} \lambda_0$, and claim that $\lambda$ satisfies equation~\eqref{integral}. Indeed, let $f \in C_c(\HH)$ and assume for now that the support of $f$ is contained in $\bar{B}$. Let $h$ be the function that is equal to $f \circ \pi$ on $\bar{\BB}$ and $0$ elsewhere. This function is continuous and its support is contained in $\bar{\BB}$ by construction. Using that the stabilizer of $\bar{\BB}$ in $\Mod(S,\Sigma)$ is also $\Gamma$, we calculate that for any $q \in \HH$, we have $\int_{\HHm} h \ d\theta_q = |\Gamma| f(q)$. We have
    \begin{align*}
    \int_B f \ d\nu &= \frac{1}{|\Gamma|} \int_{\HH} h \ d\tilde{\nu} = \frac{1}{|\Gamma|} \int_{U_{\y}} \left( \int_{\boldsymbol{L}_y} f \circ \pi \ d\nu_{\xform} \right) \ d\lambda_0(y),   
\end{align*} 
which is what we wanted. In  case the support of $f$ is arbitrary, we pick a sequence $\psi_n$ of uniformly bounded smooth functions with support contained in $\bar{\B}$ and that converge pointwise to $1_B$, the indicator function of $B$, and we apply the previous computation to $\psi_nf$ in place of $f$. We have
\begin{align*}
    \int_B \psi_n f \ d\nu = \int_{U_{\y}} \left( \int_{\boldsymbol{L}_y} \psi_n f \circ \pi \ d\nu_{\xform} \right) \ d\lambda(y).   
\end{align*} 
Passing to the limit using Lebesgue's dominated convergence, we obtain equation~\eqref{integral}.  
\end{proof}

\subsection{The special flat measures are horospherical}\label{sec: The special flat measures are horospherical}
In this subsection we prove Theorem \ref{thm: cone measures are horospherical}, which gives us our first  examples of horospherical measures. Namely we will show that the Masur-Veech measures on strata, and more generally, the special flat  measures defined in equation~\eqref{eq: special flat measure}, are horospherical.

Let $\M$ be an invariant subvariety and let $L$ be a lift of $\M$. In order to establish Theorem \ref{thm: cone measures are horospherical}, we shall first establish that the measure $m_{L}$ as in equation~\eqref{irreducible components of the cone measure} is horospherical. This will be achieved in Proposition \ref{prop: horospherical for irreducible component}. We need some preparatory results. We recall that the measure $m_{L}$ is obtained by the cone construction applied to $\mu_{L}$, i.e., for any Borel set $A \subset L^{(1)}$,
$$
m_{L}(A) = \mu_{L}(\mathrm{cone}(A)), 
$$
and the measure $\mu_{L}$ is itself obtained by integration of $\alpha = \alpha_\x \wedge \alpha_\y$, where $\alpha_\x$ and $\alpha_\y$ are as in equation~\eqref{eq: two forms}. Let $\beta \df \iota_{E} \alpha$. By construction, $\beta$ induces a volume form on $L^{(1)}$ and we denote by $\mu_{\beta}$ the measure obtained by integration of $\beta$. The following relates the measure $\mu_{\beta}$ and the cone measure $m_{L}$. 

\begin{lem}\label{cone} We have 
$$
\mu_{\beta} = 2\dim(\M) \cdot m_{L}. 
$$
\end{lem}

\begin{proof}
The proof is an application of Stokes' theorem. It follows from equations~\eqref{eq: invariance property} and \eqref{eq: euler} that the Lie derivative of $\alpha$ with respect to the Euler vector field satisfies $\mathcal{L}_E(\alpha) = 2\dim(\M) \cdot \alpha$. Let $U$ be an open set in $L^{(1)}$ contained in one chart for the manifold structure on $L$. Notice that the only part of the boundary $\partial\mathrm{cone}(U)$ of $\mathrm{cone}(U)$ to which $E$ is not tangent is $U$ itself. In particular, the only part of $\partial\mathrm{cone}(U)$ on which $\iota_E\alpha$ does not vanish is $U$. We have from Stokes' formula (for manifolds with corners, see e.g. \cite{lee2012smooth}) that
$$
\int_{\mathrm{cone}(U)} d(\iota_E \alpha) = \int_{\partial\mathrm{cone}(U)} \iota_E\alpha = \int_U \iota_E\alpha = \mu_{\beta}(U). 
$$
It follows from the Cartan formula  that $\mathcal{L}_E(\alpha) = d \iota_E \alpha + \iota_E d \alpha$. Since $\alpha$ is closed, we deduce that $d \iota_E(\alpha)= \mathcal{L}_E(\alpha)$. Gathering everything, we obtain
$$
2\dim(\M) \cdot m_{L}(U) = \mu_{\beta}(U). 
$$ 
This is true for all $U$ as above, and these open sets generate the Borel $\sigma$-algebra on $L^{(1)}$.
\end{proof}

We introduce two new vector fields $E_\x$ and $E_\y$ on $\HHm$:
\begin{align}
   E_x(\q) \df 
   \frac{\partial}{\partial t}\Big|_{t=0}
   \begin{pmatrix} e^t & 0 \\ 0 & 1 
    \end{pmatrix} 
   \cdot \q \ \ \ \text{and} \ \ \  E_y(\q) \df \frac{\partial}{\partial t}\Big|_{t=0} \begin{pmatrix} 1 & 0 \\ 0 & e^t \end{pmatrix} \cdot \q.
\end{align}
By definition we have $E = E_x + E_y$, and for any $\q \in \HHm$, we have
\begin{align}\label{eq: derivative of the geodesic flow}
\frac{\partial}{\partial t}\Big|_{t=0} g_t \cdot \q = E_x(\q) - E_y(\q).
\end{align}

For the proof of Theorem \ref{thm: cone measures are horospherical}, we will also need the following calculation: 
\begin{lem} \label{lem: 3 forms} The restrictions of the forms
  $\alpha_{\x}$ and $\alpha_{\y}$ to $L^{(1)}$ satisfy
   $$
  \iota_E (\alpha_{\x} \wedge \alpha_{\y}) = 2 (\iota_E \alpha_{\x}) \wedge \alpha_{\y}.
  $$
  \end{lem}

\begin{proof}
Let $n = \dim(\M)$. We begin by observing that on restriction to $L^{(1)}$, we have
  \begin{equation}\label{eq: observing}
    \iota_{(E_{\x}-E_{\y})}(\alpha_{\x}\wedge\alpha_{\y})=0.
\end{equation}
Indeed, we deduce from equation~\eqref{eq: derivative of the geodesic flow} that $E_x-E_y$ is tangent to $L^{(1)}$. In particular, since $L^{(1)}$ has dimension $2n-1$ (as a real vector space), any family of $2n-1$ linearly independent vector fields that are tangent to $L^{(1)}$ contain $E_\x - E_\x$ in their span. This implies equation~\eqref{eq: observing}.

Now we calculate:
\[
\begin{split}
    \iota_E(\alpha_{\x}\wedge\alpha_{\y}) &= \iota_{2E_{\x} - (E_{\x}-E_{\y})} (\alpha_{\x}\wedge\alpha_{\y})\\
    &\stackrel{\eqref{eq: observing}}{=} 2\iota_{E_{\x}}(\alpha_{\x}\wedge\alpha_{\y}) = 2\iota_{E_{\x}}\alpha_{\x} \wedge\alpha_{\y} + (-1)^{n} \alpha_\x \wedge\iota_{E_x} \alpha_\y.
\end{split}
\]
The last equality follows from the Leibniz formula for contractions 
\[
\iota_V(\alpha\wedge\beta)=(\iota_V\alpha)\wedge\beta+(-1)^{\deg(  \alpha)}\alpha\wedge\iota_V\beta.
\]
Now, notice that $E_\x$ is tangent to the fibers of $\pi_\y \circ \dev$. Since $\alpha_\y = (\pi_\y \circ \dev)^{\ast} \eta_\y$, we deduce that $\iota_{E_\x} \alpha_\y = 0$. Similarly, we prove that $\iota_{E_{\y}}\alpha_{\x} = 0$ and thus 
$$
\iota_E(\alpha_{\x}\wedge\alpha_{\y}) = 2\iota_{E_\x} \alpha_{\x} = 2\iota_{E} \alpha_{\x}.
$$ 
\end{proof}

\begin{prop}\label{prop: horospherical for irreducible component}
    The measure $m_{L}$ is horospherical. 
\end{prop}

\begin{proof}
It follows from Lemma \ref{cone} that $m_{L}$ is given, up to a multiplicative constant, by integration of the differential form $\beta = \iota_E( \alpha_\x \wedge \alpha_\y)$. Lemma \ref{lem: 3 forms} implies that $\beta = 2 \beta_\x \wedge \alpha_\y$.  Notice that both the forms $\beta_\x$ and $\alpha_\y$ are {\em basic}, i.e., they are obtained by pullback of forms on $\mathbf{P}(V_\x)$ and $V_\y$ by the projections $\dev \circ \pi_\x'$ and $\dev \circ \pi_\y$. Indeed, we have $\alpha_\y = (\dev \circ \pi_\y)^{\ast} \eta_y$ and using Lemma \ref{lem: local diffeo}, we can build a differential form $\beta_\x'$ on $\mathbf{P}(V_\x)$ such that $(\dev \circ \pi_\x')^{\ast} \beta_\x' = \beta_x$.

Now, let $\varphi : \product \to \BB$ be a box in $L^{(1)}$ and let $f \in C_c(L)$. Notice that $\varphi^{\ast} \alpha_\y = \eta_\y$ and $\varphi^{\ast}\beta_\x = \beta_\x'$. We have: 
\begin{align*}
\int_{\BB} f \cdot \beta_\x \wedge \alpha_\y &= \int_{\product} f \circ \varphi \cdot \beta_\x' \wedge \eta_\y \\
&=  \int_{U_{\y}} \Big( \int_{U_\x' \times \{y\}} f \circ \varphi \cdot \beta_\x' \Big) \cdot \eta_\y \\
&=  \int_{U_{\y}} \Big( \int_{\LL_y} f \cdot \beta_\x \Big) \cdot \eta_\y.
\end{align*}
If we let $\lambda$ be the measure on $U_\y$ given by integration of the form $\frac{2}{\dim(\M)} \cdot \eta_y$, we obtain 
$$
\int_{\BB} f \ dm_{L} = \int_{U_{\y}} \Big( \int_{\LL_y} f \ d\xmeasure \Big) \ d \lambda(y).
$$
\end{proof}

\begin{proof}[Proof of Theorem \ref{thm: cone measures are
    horospherical}]

By definition, in order to prove that $m_{\M}$ is horospherical, we need to show that the lifts of $m_{\M}$ are horospherical. We recall from equation~\eqref{irreducible components of the cone measure} that the lifts of $m_{\M}$ are given by the $m_{L}$. Theorem \ref{thm: cone measures are horospherical} is then a consequence of Proposition \ref{prop: horospherical for irreducible component}.  
\end{proof}

\subsection{The horocycle flow, real REL, and horospherical measures}
In this subsection we will show that the horocycle flow and the real Rel deformations move points in their horospherical leaf, and preserve horospherical measures.

\begin{prop}\label{prop: horocycle invariance}
Horospheres and horospherical measures are horocycle flow-invariant. 
\end{prop}

\begin{proof}
 For $\q \in \HH_{\mathrm{m}}$ with $\dev(\q) = (x, y)$ and $s \in \R$, we have 
  \begin{equation}\label{action of us}
     \dev(u_s \q) = (x + sy, y).
 \end{equation}
 This implies that the horocycle flow maps $\WW^{uu}(\q)$ to itself and since $\pi$ is $\GL$-equivariant, we deduce that horospheres are preserved by the horocycle flow.  

We also deduce from equation~\eqref{action of us} that the horocycle flow preserves the form $\alpha_\x$. Since the horocycle flow commutes with the rescaling action, it also preserves the Euler vector field, from which we deduce that for any $s \in \R$ we have ${u_s}^{\ast} \xform = \xform$. In particular, the horocycle flow preserves the measures $\nu_{\xform}$.

Let $\nu$ be a horospherical measure and let $s_0 >0$. We claim that for any $q \in \M$ such that the orbit segment $\{u_sq : s\in \R, \ |s| \leq s_0 \}$ is embedded, i.e., $Uq$ is not a periodic horocycle orbit with period of length smaller than $2s_0$, there is an open set $\U \subset \M$ containing $q$ such that for any compactly supported continuous function $f$ with support contained in $\mathcal{U}$ and any $s \in \R$ with $|s| \leq s_0$,  we have:  
\begin{equation}\label{eq: this gives}
    \int_{\M} f\circ u_{s} \ d\nu = \int_{\M} f \ d\nu. 
\end{equation}
To see this, let $\q \in \HHm$ be such that $\pi(\q)=q$ and define
$$
\boldsymbol{\sigma} \df \{ u_s \q : s \in \R, \ |s| \leq s_0 \}.
$$

Let $\Gamma$ be the stabilizer in $\Mod(S,\Sigma)$ of $\q$. Since the $\GL$-action on $\HHm$ commutes with $\Mod(S,\Sigma)$, the group $\Gamma$ acts trivially on $\boldsymbol{\sigma}$ and since $Uq$ is not a periodic orbit with period smaller than $2s_0$, we have that that $\boldsymbol{\sigma} \cdot \gamma \cap \boldsymbol{\sigma} = \emptyset$ for any $\gamma \in \Mod(S,\Sigma)$, unless $\gamma \in \Gamma$. By thickening $\boldsymbol{\sigma}$, we can find a box $\BB \subset \pi^{-1}(\M)$ containing $\boldsymbol{\sigma}$ and up to replacing $\BB$ with  $\cap_{\gamma \in \Gamma} \BB \cdot \gamma$, we can assume that $\BB$ is regular. By construction, for any $s \in \R$ with $|s| \leq s_0$, the surface $u_sq$ belongs to $B = \pi(\BB)$ and lies on the plaque of $q$. Since the horocycle flow acts continuously, there is a neighborhood $\U \subset \B$ around $q$ such that for any $\q' \in \pi^{-1}(\U) \cap \BB$ and $|s| \leq s_0$, we have $u_s \q' \in \BB$. Let $\lambda$ be a transverse measure on $U_{\y}$, i.e., a measure as in equation~\eqref{integral}, and let $f$ be a continuous function with support contained in $\mathcal{U}$. It follows from equation~\eqref{integral} that
\begin{align*}
    \int_{\M} f \circ u_{s} \ d \nu &= \int_{U_{\y}}  \left(
    \int_{\boldsymbol{L}_y} f \circ u_{s} \circ \pi \ d \nu_{\xform} \right) \ d\lambda(y)\\ 
    &= \int_{U_{\y}}  \left( \int_{\boldsymbol{L}_y} f  \circ \pi \ d {u_s}_{\ast} \nu_{\xform} \right) \ d\lambda(y)\\ 
    &= \int_{U_{\y}}  \left( \int_{\boldsymbol{L}_y} f  \circ \pi \ d  \nu_{\xform} \right) \ d\lambda(y) = \int_{\M} f \ d\nu.
\end{align*}
The second equality follows from the fact that $\pi$ is $\GL$-equivariant and the fact that for any $y \in U_{\y}$, the action of $u_s$ maps $\mathrm{supp}(f \circ \pi) \cap \boldsymbol{L}_y$ inside $\boldsymbol{L}_y$, which in turns is implied by our choice of $\U$ and the fact that the horocycle flow maps the leaves of $\WW^{uu}$ into themselves.

For any $f \in C_c(\HH)$, define
$$
s_f \df \inf \{s >0 : \mathrm{supp}(f) \ \text{contains a periodic surface of period} \ s \}.
$$
It is easy to see that $s_f$ is always positive, and using the first part of the proof together with a partition of unity argument, we can show that equation~\eqref{eq: this gives} holds for $f$ and $s \in \R$ with $|s| \leq s_f$. Furthermore, notice that for any $s \in \R$, we have $s_{f} = s_{f \circ u_s}$. Writing $s=ks'$ with $k \in \mathbb{N}$ and $|s'| \leq s_f$, we obtain
$$
\int_{\M} f \circ u_{s} \ d\nu = \int_{\M} f \circ u_{(k-1)s'} \ d\nu = \cdots = \int_{\M} f \ d\nu.
$$
This proves that $\nu$ is horocycle flow-invariant. 
\end{proof}

For any irreducible invariant subvariety $\M$, we let 
\begin{equation}\label{eq: def ZM}
    Z_{\M} \df V \cap Z,
\end{equation}
where $V$ is the model space of some lift of $\pi^{-1}(\M)$ and $Z$ is the real REL space. Notice that the space $Z_{\M}$ actually does not depend on the choice of particular lift. This is a consequence of Proposition \ref{characterization of irreducible} together with with the fact that $\Mod(S,\Sigma)$ acts trivially on $\ker(\Res)$. 

\begin{prop}\label{prop: real REL invariance}
Let $v \in Z_{\M}$, $q \in \M$, and suppose  the Rel flow $\Rel_v(q)$ is defined. Then $\Rel_v(q) \in W^{uu}(q)$. If $\nu$ is a horospherical measure and $\Rel_v(q)$ is defined for $\nu$-a.e.\;$q$, then $\nu$ is invariant  under the (almost everywhere defined) map $q \mapsto \Rel_v(q)$.
\end{prop}

\begin{proof}
Let $Z^{(q)}$ be as in equation~\eqref{eq: set where defined}. Since $Z_{\M} \subset V$, where $V$ is the subspace that $\M$ is modeled on, we have $\Rel_{v}(q) \in \M$ if $q \in \M$ and $v \in Z_{\M} \cap Z^{(q)}$. The only properties of the horocycle flow which were used in the proof of Proposition \ref{prop: horocycle invariance} are that $u_{s_0}$ preserves the horospheres, and  acts on them by translations. The same properties are valid for the action of $\Rel_v$ for $v \in Z_{\M}$. Indeed, $\Rel_v$ sends surfaces of area one to surfaces of area one, and if $\dev (\q) = (x,y)$ then $\dev (\Rel_v \q) = (x + v , y)$. 
\end{proof}

 If a measure $\mu$ on $\M$ is saddle-connection free, then for $\mu$-a.e. $q \in \M$, $\Rel_v(q)$ is defined for every $v \in Z_\M,$ and satisfies the `group law' property 
$$\forall v_1, v_2 \in Z_\M, \ \ \Rel_{v_1}\left( \Rel_{v_2}(q) \right) = \Rel_{v_1+v_2}(q).$$
Following \cite{Wright_cylinders}, we say that an irreducible invariant subvariety $\M$ {\em is of rank one} if $\dim(\Res(V))=2,$
where $V$ is the model space of any lift of $\pi^{-1}(\M)$. In the rank one case we have the following converse to Propositions \ref{prop: horocycle invariance} and \ref{prop: real REL invariance}: 
\begin{prop}
    If $\M$ is an invariant subvariety of rank one and $\mu$ is a saddle connection-free measure, then $\mu$ is horospherical if and only if it is invariant under the horocycle and the real Rel flows. 
\end{prop}

\begin{proof}
By a dimension count, we see that when $\M$ has rank one, the dimension of horospheres is the same as $\dim(Z_\M)+1.$ This means that the horosphere $W^{uu}(q)$ satisfies 
$$
W^{uu}(q) = \{\Rel_v(u_sq): s \in \R, \ v \in Z_\M\},$$
that is the group action generated by the  horocycle flow and real Rel acts transitively on the horospheres. 
As we saw in the proofs of Propositions \ref{prop: horocycle invariance} and \ref{prop: real REL invariance}, this action is by translations, with respect to the affine structure on $W^{uu}(q)$ afforded by Lemma \ref{lem: orbifold foliation}. Since the measures $\nu_{\beta_{\x}}$ are the unique (up to scaling) translation-invariant  measures on the affine manifolds $L^{(1)},$ the invariance of $\mu$ under the horocycle and real Rel flows implies that the conditional measures on the plaques in a box are given by $\nu_{\beta_{\x}}.$
    \end{proof}

\subsection{Further properties}
Let $X$ be a manifold with a foliation, and a Borel measure $\mu$. We say that $\mu$ is {\em ergodic for the foliation} if any Borel subset $B$ which is a union of leaves satisfies either $\mu(B)=0$ or $\mu(X \sm B)=0$. For instance we have:

\begin{prop}
The special flat measure on an invariant subvariety is ergodic for the horospherical foliation.
\end{prop}

\begin{proof}
This follows from Proposition \ref{prop: horocycle invariance}, the ergodicity of the special flat measure with respect to the $A$-action (see e.g. \cite[Section 4]{Forni-Matheus_survey}), and the Mautner phenomenon (see \cite{Einsiedler_Ward}).
\end{proof}

Denote by  $\mathcal{P}^{(uu)}\left(\M\right)$ the collection of horospherical measures on $\M$ with total mass at most one. The following standard results in ergodic theory are valid in the context of horospherical measures:

\begin{prop}\label{prop: ergodic decomposition}
For the horospherical foliation on any invariant subvariety $\M^{(1)}$, we have:
\begin{enumerate}
    \item
    The space $\mathcal{P}^{(uu)}\left(\M\right)$, with the weak-$*$ topology, is a compact convex set. 
    \item A horospherical probability measure is ergodic if and only if it is an extreme point of  $\mathcal{P}^{(uu)}\left(\M\right)$. 
    \item For any probability measure $\mu \in \mathcal{P}^{(uu)}\left(\M\right)$ there is a probability  space $(\Theta, \eta)$ and a measurable map $\Theta \to \mathcal{P}^{(uu)}\left(\M\right), \ \theta \mapsto \nu_\theta$, such that $\nu_\theta$ is ergodic and a probability measure for $\eta$-a.e. $\theta$, and $\mu = \int_{\Theta} \nu_\theta \, d\eta(\theta)$.
    \item If $\mu_1, \mu_2 \in \mathcal{P}^{(uu)}\left(\M\right)$ such that $\mu_1 \ll \mu_2$, and $\mu_2$ is  ergodic, then $\mu_1 = c\mu_2$ for some $c \geq 0$. 
\end{enumerate}
  
\end{prop}

\begin{proof}
It is clear from definitions that if $\mu_1, \mu_2 \in \mathcal{P}^{(uu)}\left(\M\right)$, and $\alpha \in (0,1)$, then $\alpha \mu_1 +(1- \alpha)\mu_2 \in \mathcal{P}^{(uu)}\left(\M\right)$. It is also clear that condition \eqref{integral}, and the condition $\mu(\M) \leq 1$, are both closed conditions. This proves the first assertion. The remaining assertions follow from Choquet's theorem by standard arguments, see e.g. \cite[Chap. 2.6]{Candel_Conlon}. 
\end{proof}

We say that two surfaces $q, q' \in \HH$ are {\em horizontally equivalent} if there is a homeomorphism $ M_{q} \to M_{q'}$ of the underlying surfaces that preserves the labels of singularities and maps the union of the horizontal saddle connections of $M_q$ bijectively to the union of those of $M_{q'}$.  Note that a horizontal equivalence only preserves {\em certain} horizontal structure. It preserves saddle connections but need not preserve the horizontal foliation.
  
\begin{prop}\label{prop: same saddle connections}
Any two surfaces in the same horospherical leaf are horizontally equivalent. 
\end{prop}

\begin{proof}
It suffices to show this upstairs; that is, we let $\q, \q' \in \HH_{\mathrm{m}}$ with $\bq' \in \WW^{uu}(\q)$, let $f: S \to M_q$ and $f': S \to M_{q'}$ be marking maps representing $\q, \q'$ and show that, if $f$ and $f'$ are carefully chosen, $f' \circ f^{-1}$ gives a bijection of  horizontal saddle connections. We first discuss $\q' \in \WW^{uu}(\q)$ which are sufficiently close to $\q$. Let $f: S \to M_q$ be a marking map representing $\q$ and let $\sigma_1, \ldots, \sigma_r$ be the horizontal saddle connections on $M_q$. Let $\tau$ be a triangulation of $S$ such that the segments $f^{-1}(\sigma_i)$ are edges of triangles. Let $f'$ be constructed from $\tau$, so that the map $f' \circ f^{-1}$ is affine on each triangle of $\tau$ (see the discussion of comparison maps in \cite[\S 2.4]{BSW}). Let $\mathcal{U} = \mathcal{U}_\tau$ be the corresponding neighborhood of $\q$. Then for any surface $\q' \in \mathcal{U}$, represented by $f': S \to M_{q'}$, the paths $f'\circ f^{-1}(\sigma_i)$ are represented by saddle connections on $M_{q'}$. Furthermore, if $\q' \in \mathcal{U} \cap \WW^{uu}(\q)$ then these paths are horizontal saddle connections, so that $f' \circ f^{-1}: M_q \to M_{q'}$ is a homeomorphism mapping the horizontal saddle connections of $M_q$ injectively to horizontal saddle connections on $M_{q'}$.

Now choose $\q_{\max} \in \WW^{uu}(\q)$ so that it has the maximal number of horizontal saddle connections. We will show that the set $\mathcal{V}$ of surfaces in $\WW^{uu}(\q)$ which are horizontally equivalent to $\q_{\max}$ is open and closed, and this will conclude the proof. By the preceding discussion, $\mathcal{V}$ is open in $\WW^{uu}(\q)$. Furthermore, if $\q_n \to \q_\infty$ is a convergent sequence of surfaces in $\WW^{uu}(\q)$, with $\q_n \in \mathcal{V}$, then the horizontal saddle connections on the surfaces $M_{q_n}$ have length bounded uniformly from above and below, and so converge to paths on $M_{q_\infty}$ which are represented by horizontal saddle connections or by finite concatenations of horizontal saddle connections. Sequences of paths which are distinct on the surfaces $M_{q_n}$ cannot converge to the same paths on $M_{q_\infty}$ because they issue from different singularities, or from different prongs at the same singularity. Thus $\q_\infty$ has at least the same number of saddle connections as $\q_{\max}$, and so, by maximality, $\q_\infty \in \mathcal{V}$. This completes the proof. 
\end{proof}

From Proposition \ref{prop: same saddle connections} we deduce: 
  
\begin{cor}\label{cor: equivalent diagrams}
If $\nu$ is an ergodic horospherical measure then there is a subset $\M' \subset \M$ of full $\nu$-measure such any two surfaces in $\M'$ are horizontally equivalent. 
\end{cor}

\begin{rem}
In \cite[Def. 5.1]{BSW}, using boundary marked surfaces, {\em topological horizontal equivalence} is  introduced. In this definition the homeomorphism $M_q \to M_{q'}$ is required to preserve additional structure, e.g. the angular differences between saddle connections at each singular point. Proposition \ref{prop: same saddle connections} and Corollary \ref{cor: equivalent diagrams} hold for this finer notion of equivalence as well.  
\end{rem}

\section{Saddle connection free horospherical measures}\label{sec: main proof} 
In this section we will prove Theorem \ref{thm: classification}. We first state and prove some auxiliary statements.

\subsection{The Jacobian distortion in a box}\label{subsec: distortion}

The different plaques in a box can be compared to each other using the structure of a box. Namely, let $\varphi: \product \to \BB \subset L^{(1)}$ be a box. For any point $y \in U_{\y}$ we define
$$
\varphi_{y}: U'_{\x} \to \LL_y, \ \ \ \ \varphi_y([x] ) \df  \varphi([x],y),
$$ 
where $\LL_y$ is the plaque of $y $ in $\BB$ (see Definition \ref{def: box}). 

For any two points $y_0$ and $y_1$ in $U_{\y}$, the map $\varphi_{y_0,y_1} \df \varphi_{y_1} \circ \varphi_{y_0}^{-1}$ is a diffeomorphism between the plaques $\LL_{y_0}$ and $\LL_{y_1}$ in $B$, identifying points parameterized by the same point in $U'_{\x}$. Define
$$
\delta_{y_1} : \LL_{y_0} \to \R, \ \ \ \delta_{y_1} (\q ) \df  \langle x_{\q}, y_1 \rangle^{-\dim(\M)},
$$
where $x_{\q} = \pi_\x \circ \dev(\q)$. The diffeomorphism $\varphi_{y_0,y_1}$ is not measure preserving. Instead, we have the following: 

\begin{prop}\label{jacobian} (Jacobian calculation) 
For any two points $y_0, y_1 \in U_{\y}$ we have
\begin{equation}\label{eq: appearing in}
(\varphi_{y_0,y_1})^{\ast} \left(\xform|_{\LL_{y_1}} \right) = \delta_{y_1} \cdot \left(\xform|_{\LL_{y_0}} \right). 
\end{equation}

\end{prop}

\begin{proof} 
For any $y \in U_{\y}$, write 
$$
\bar{L}_y \df \pi_{\x} \circ \dev (\LL_y),
$$
where $\pi_{\x}$ is the projection in equation~\eqref{eq: def projections}. Then $\bar{L}_y$ is an open subset of the affine hyperplane
\begin{equation}\label{eq: translate of}
  \{x \in V_{\x}: \langle x, y \rangle =1\}. 
\end{equation}
By Definition \ref{def: box} the map
$$
F: \bar{L}_y \to \LL_y, \ \ \ \ F(x ) \df  \varphi([x], y)
$$
is a diffeomorphism with inverse $\pi_{\x} \circ \dev$. We denote by $e_\x$ the Euler vector field on $V_\x$. Notice that $ F^* \beta_{\x}$ is the restriction to $\bar L_y$ of $\iota_{e_{\x}}\eta_\x$. Indeed, we calculate
\[\begin{split}
    F^* \beta_{\x} = F^* \iota_E (\pi_{\x} \circ \dev^* \eta_{\x})  =  F^* (\pi_{\x} \circ \dev) ^* \left(\iota_{e_{\x}} \eta_{\x} \right) = \iota_{E_{\x}} (\eta_{\x}).
\end{split}\]
The map $F$ gives a chart of $\LL_y$ in which $\beta_\x$ is  $\iota_{e_\x} \eta_\x$.  We shall perform our calculation in these charts and verify equation~\eqref{eq: appearing in} in $\bar{L}_y$ instead of $L_y$. Let $y_0, y_1 \in U_{\y}$, and set
$$
h: \bar{L}_{y_1} \to \R, \ \ \ \ \ \ h(x) \df \frac{1}{\langle x, y_1 \rangle}.
$$
The map $\varphi_{y_0, y_1} : L_{y_0} \to L_{y_1}$ is expressed in charts simply as the map 
$$
\bar{\varphi}_{y_0, y_1}: \bar{L}_{y_0} \to \bar{L}_{y_1}, \ \ \
\bar{\varphi}_{y_0, y_1}(x) =  h(x) \, x. 
$$
This implies by the product rule that
$$
(D \bar{\varphi}_{y_0, y_1})_x(v) = h(x) \, v + (Dh)_x(v) \, x.
$$
Hence, denoting $d = \dim(\M)$, for $v_1, \ldots, v_{d-1}$ in the tangent space to $\bar{L}_{y_0}$ at $x$  we have:
\[\begin{split}
    &\left((\bar{\varphi}_{y_0, y_1})^* \iota_{e_\x} \eta_x \right)_{x} (v_1,\ldots , v_{d-1})  \\ = & (\iota_{e_\x} \eta_x)_{\bar{\varphi}_{y_0, y_1}(x)}\Big(D_x\bar{\varphi}_{y_0, y_1}(v_1), \ldots, D_x\bar{\varphi}_{y_0, y_1}(v_{d-1}) \Big) \\
    = &(\eta_\x)_{h(x)x} \Big (h(x)x, h(x)v_1 + D_xh(v_1) \, x, \ldots, h(x)v_{d-1}+ D_xh(v_{d-1}) \, x\Big)\\ 
    =&(\eta_\x)_{h(x)x} \Big (h(x)x, h(x)v_1,\ldots,h(x)v_{d-1} \Big ) =h(x)^d (\iota_{e_\x} \eta_\x)_x(v_1,\ldots,v_{d-1}).
 \end{split}              
\]
This is Formula  \eqref{eq: appearing in}. 
\end{proof}

Notice that for any $y_0,y_1 \in U_{\y}$ and $[x] \in U_{\x}'$, we have 
$$
\delta_{y_1} \circ \varphi_{y_0}([x]) = \Big( \frac{\langle x, y_0\rangle}{\langle x, y_1\rangle} \Big)^{\dim(\M)}.
$$
This leads us to define the {\em distortion of $\BB$} as follows: 
$$
\delta_{\BB} \df \mathrm{sup} \left\{ \left|1 - \left(\frac{\langle x, y_0 \rangle }{\langle x, y_1 \rangle} \right) ^d \right| : [x] \in U'_{\x}, \  y_0, y_1 \in U_{\y} \right\}. 
$$ 

\begin{rem} 
The quantity $\frac{\langle x, y_0\rangle}{\langle x, y_1\rangle}$
has the following geometric interpretation. The points $\varphi([x],y_0)$ and $\varphi([x],y_1)$ are in the same weak stable leaf, and this leaf is further foliated by strong stable leaves. The geodesic flow maps a given weak stable leaf to itself, permuting the strong stable leaves inside it. The choice $t = \log \left(\frac{\langle x, y_0\rangle}{\langle x, y_1\rangle} \right)$ is the value of $t\in \R$ for which $g_t$ maps $\varphi([x], y_0)$ to the strong stable leaf of $\varphi([x], y_1)$.  
\end{rem}

The distortion can be used to bound the variation of the mass of the horospherical plaques of $\BB$ with respect to the measures $\nu_{\xform}$. Indeed, by an easy change of variables, using $\varphi_{y_0,y_1}$ we have  
\begin{align*}
    \left  | \xmeasure (\LL_{y_1}) - \xmeasure(\LL_{y_0}) \right| \leq \delta_{\BB}\, \xmeasure(\LL_{y_0}).
\end{align*}
From this it follows that
\begin{equation}\label{variationofmass}
    \left| \frac{\xmeasure(\LL_{y_1})}{\xmeasure(\LL_{y_0})} - 1 \right| \leq \delta_{\BB}.
\end{equation}

The distortion of a box is well-behaved with respect to the geodesic flow. For $t \in \R$ and any $\BB$ in $\pi^{-1}(\M)$, we write
$$\BB_t \df g_t(\BB) \ \ \ \text{ and } \ \ B_t \df \pi(\BB_t).$$

\begin{prop}\label{prop: pushofbox1}
Let $\BB$ be a box in $\pi^{-1}(\M)$. Then $\BB_t$ is a box with $\delta_{\BB_t} = \delta_{\BB}$ and it is regular whenever $\BB$ is.  
\end{prop}

\begin{proof} Let $\varphi: U'_{\x} \times U_{\y} \to L^{(1)}$ be the parametrization of $\BB$, where $L$ is an irreducible component of $\pi^{-1}(\M)$ and let $\bar{U}_{\y}$ be the image of $U_{\y}$ under multiplication by $e^{-t}$, and let $\bar{\varphi} ([x], y) \df g_t \circ \varphi ([x], e^ty)$. Using the fact that the geodesic flow preserves the splitting into stable and horospherical foliation, and acts on $V_{\y}$ by multiplication by $e^{-t}$, we see that $\bar{\varphi}: U'_{\x} \times \bar{U}_{\y} \to L^{(1)}_1$ is a parameterization of $\BB_t$ as in Definition \ref{def: box}. Also, for $i=0,1$,  if $([x], \bar{y}_i) \in U'_{\x} \times \bar{U}_{\y}$, where $\bar{y}_i = e^{-t}y_i$, then
$$
\frac{\langle x, \bar{y}_0 \rangle}{\langle x, \bar{y}_1 \rangle} =
\frac{\langle x, y_0 \rangle}{\langle x, y_1 \rangle}. 
$$

This implies that the distortion of $\BB$ is the same as the distortion of $\BB_t$. 

The last statement follows from the fact that the actions of $\GL$ and $\Mod(S,\Sigma)$ commute. 
\end{proof}

\subsection{Thickness of a box} 
We now introduce the notion of the thickness of a box.
To define this quantity we use the sup-norm Finsler metric of \S \ref{subsec: AGY metric} to induce a distance function on leaves of the stable foliation. We rely on work of Avila and Gouezel \cite[\S 5]{avila2010small}, who defined a similar distance function on the leaves of the strong stable foliation.

For a subset of a stable leaf, we denote by $\mathrm{diam}^{(s)}$ its diameter with respect to the distance function $\mathrm{dist}^{(s)}$.  We define the {\em thickness} of the box $\BB$ as
$$
\tau_{\BB} \df \underset{[x] \in U'_{\x}}{\sup} \mathrm{diam}^{(s)} \ \varphi( \{[x]\} \times U_{\y}) ;
$$
that is, the maximal diameter of a plaque for the stable foliation. 

We will need boxes whose thickness is also well-behaved under the geodesic flow. Similarly to Proposition \ref{prop: pushofbox1}, we have: 
\begin{prop}\label{prop: pushofbox2}
For any $\vre >0$ and any $\q \in \pi^{-1}(\M^{(1)})$, there is a regular box $\BB$ in $\pi^{-1}(\M)$ containing $\q$ such that for any $t\geq 0$, $\tau_{\BB_t} \leq \vre.$  
\end{prop}

\ignore{
\textcolor{blue}{In the thickness discussion we would like to know
  that under the forward geodesic flow lengths of paths in the leaves
  of the weak stable foliation do not grow when we apply the geodesic
  flow. AG prove this for the strong stable foliation. (Recall the
  confusing notation remark.) We also know (AG) that the lengths of
  geodesic paths do not grow. To connect two points in a weak stable
  leaf we can concatenate a weak stable path and a geodesic path. The
  length of this path is an upperbound for distance. If our boxes have
  controlled diameter in both the strong stable direction and the
  geodesic direction we can show that the lengths of weak stable paths
  do not grow under the geodesic flow. Observe that the ``geodesic
  width" is bounded by $|\log \frac{\langle x, y_0  \rangle }{\langle
   x,  y_1  \rangle}|$ (see previous ``geometric interpretation") so in
  fact we have already bounded it in the discussion giving a bound for
  $\delta$.} 

\textcolor{blue}{Connection between $d^s$ and $d$: $d\le d^s$ and if
  both are small they are comparable in size for points in the strong
  stable manifold. It follows from AG Prop. 5.3 that distance less
  than 1/25 implies distortion less than 2.} 
\textcolor{red}{add the above discussion to the proof of Prop. 4.2.}

}
\begin{proof}
Let $L$ be a lift of $\M$ that contains $\q$ and let $\Gamma$ be the stabilizer in $\Mod(S,\Sigma)$ of $\q$. Since $\Mod(S,\Sigma)$ acts properly discontinuously on $\HHm$, there is a neighborhood $\V$ containing $\q$ such that for any $\gamma \in \Mod(S,\Sigma)$, either $\V \cdot \Gamma \cap \V = \emptyset$ or $\gamma \in \Gamma$. By Lemma \ref{lem: local diffeo}, let $\bar \BB \subset \V$ be a box containing $\q$ and let $\bar{\varphi} :  \bar{U}'_{\x} \times \bar{U}_{\y} \to \bar{\BB}$ be the parametrization of $\bar{\BB}$. Let $\dev(\bq) = (x_0, y_0)$, let $\hat{U}'_{\x}$ be a neighborhood of $[x_0]$ whose closure is contained in $\bar{U}'_{\x}$, and let

$$
\mathbf{C} \df \overline{\varphi} \left(\hat{U}'_{\x}  \times \{y_0\}  \right).
$$
That is, $\mathbf{C}$ is a bounded subset of a horospherical leaf, contained in a plaque of $\bar{\BB}$, and with closure in the interior of $\bar{\BB}$. Let $\vre_1 \in \left (0, \frac{\vre}{4}\right)$ be small enough so that 
$$
\mathbf{C}_1 \df \bigcup_{|t| \leq \vre_1}g_t (\mathbf{C}) \subset \bar {\BB},
$$
and let $\vre_2 \in \left(0, \frac{\vre}{4} \right)$ such that 
$$
\mathbf{C}_2 \df \bigcup_{\q_1 \in \mathbf{C}_1} \left\{ \q_2 \in \WW^{ss}(\q_1)  : \mathrm{dist}^{(ss)}(\q_1, \q_2) < \vre_2 \right\} 
$$
is contained in $\bar \BB$. Such numbers $\vre_1, \vre_2$ exist because $\mathbf{C}$ is bounded, and $\mathbf{C}_2$ contains a neighborhood of $\q$. We can therefore let $U'_{\x} \subset \hat{U}'_{\x}$ and $U_{\y}\subset \bar{U}_{\y}$ be small enough open sets so that $\BB = \bar{\varphi} (U'_{\x} \times U_{\y})$ contains $\q$ and is contained in $\mathbf{C}_2$. Since $\BB$ is contained in $\V$, we may replace $\BB$ by $\cap_{\gamma \in \Gamma} \BB \cdot \gamma$ and we can assume that $\B$ is regular, with stabilizer $\Gamma$. 

For $\q \in \BB$, let $\LL^s(\q)$ be the plaque through $\q$ for the weak stable foliation, that is, the connected component of $\q$ in $\BB \cap \WW^{s}(\q)$. For each $\q_2 \in \BB$ there is a point $\q_0$, which is the unique point in the intersection $\mathbf{C} \cap \LL^s(\q_2)$, and a path from $\q_0$ to $\q_2$ which is a concatenation of two paths $\gamma_1$ and $\gamma_2$. The path $\gamma_1=\{g_t\q_0 : t \in I \}$ from $\q_0$ to $\q_1$ goes along a geodesic arc, where $I$ is an interval of length at most $\vre_1$. The path $\gamma_2$ from $\q_1$ to $\q_2$ has sup-norm length at most $\vre_2$ and is contained in $\WW^{ss}(\q_2)$. Since $\vre_1, \vre_2 < \frac{\vre}{4}$, each point in any stable plaque in $\BB$ is within distance at most $\frac{\vre}{2}$ from the unique point at the  intersection of this plaque with $\mathbf{C}$, where the distance is measured using the distance function $\mathrm{dist}^{(s)}$. Concatenating such paths we see that the diameter of any stable plaque in $\BB$ is at most $\vre$, and this implies the same bound for stable plaques in $B$. That is, the thickness of $B$ is less than $\vre$. By Proposition \ref{prop: properties of restricted AGY metric}, the lengths of geodesic paths and of paths in strong  stable leaves, do not increase when pushed by $g_t$ for $t \geq 0$. Thus the same argument (using the pushes of $\gamma_1$ and $\gamma_2$ by $g_t$) give the required upper bound on the thickness of $\BB_t$. 
\end{proof}

For a compactly supported continuous function $f$ on $\M$, we denote by $\omega_f$ its continuity modulus with respect to the sup-norm distance function. In particular, $\omega_f(t) \to 0$ as ${t \to 0+}$ and 
$$
|f(q_1) - f(q_2)| \leq \omega_f(\mathrm{dist}(q_1, q_2)) \ \ \ \text{ for any } q_1, q_2 \in \M.
$$

The following key lemma says that for any horospherical measure $\nu$, any regular box $\BB$ and any test function $f$, the integral of $f$ with respect to $\nu|_B$ can be approximated by the integral of $f \circ \pi$ with respect to $\xmeasure$ on any one horospherical plaque of $\BB$, provided that $\BB$ has small distortion and small thickness.  We recall that $B \subset \M$ is defined as the image of $\BB$ by $\pi$. 

\begin{lem}\label{comparison}
Let $\nu$ be a horospherical measure, let $f \in C_c \left(\M^{(1)} \right)$ and let $\BB$ be a regular box such that $\nu(B)>0$. Then for any $y \in U_{\y}$,

\begin{equation*}
    \left| \frac{1}{\nu(B)}\int_{B} f \, d\nu - \frac{1}{\xmeasure(\LL_{y})} \int_{\LL_{y}} f \circ \pi \ d\xmeasure \right| \leq  \omega_f(\tau_{\BB}) + 2\|f\|_{\infty}\delta_{\BB}.
  \end{equation*}
\end{lem}

\begin{proof}
For $y, y'\in U_{\y}$, let $\varphi_{y, y'} : \LL_y \to \LL_{y'}$ be as in \S \ref{subsec: distortion}. On the one hand, for any $y' \in U_{\y}$ we have
\begin{align*}
    \left| \int_{\LL_{y'}} f \circ \pi \, d \xmeasure - \int_{\LL_{y'}} f \circ \pi \circ \varphi_{y',y} \, d\xmeasure \right|  &\leq  \int_{\LL_{y'} } \left| f \circ \pi -f \circ \pi \circ \varphi_{y',y} \right| \, d\xmeasure \\
   &\leq  \omega_f(\tau_{\BB})\xmeasure(\LL_{y'}).
\end{align*}

The second inequality follows from the fact that, by definition of the thickness, for any $[x] \in U'_{\x}$, the distance between the points $\varphi([x])$ and $\varphi(\varphi_{y',y}([x]))$, with respect to the distance function $\mathrm{dist}^{(s)}$, is at most $\tau_{\BB}$ and thus also with respect to the distance function $\mathrm{dist}$, together with the fact that $\pi$ is a contraction.

On the other hand, by the definition of $\delta_{\BB},$ we have:  

\begin{align*}
    \left| \int_{\LL_{y'}} f \circ \pi \circ \varphi_{y',y} \, d \xmeasure - \int_{\LL_{y}} f \circ \pi \, d\xmeasure \right| &=  \left| \int_{\LL_{y}} f \circ \pi \, d\varphi_{y,y'}^{\ast} \xmeasure - \int_{\LL_{y}} f \circ \pi \, d\xmeasure \right| \\ 
    &\leq \|f\|_{\infty}\delta_{\BB}\xmeasure(\LL_{y}).
\end{align*}

The last inequality follows from Proposition \ref{jacobian} and the definition of $\delta_{B}$. Using  equation~\eqref{variationofmass} we deduce that for any $y,y'\in V_{\y}$,
\begin{equation*}
   \left| \frac{1}{\xmeasure(\LL_{y'})} \int_{\LL_{y'}} f \circ \pi \, d\xmeasure  -
     \frac{1}{\xmeasure(\LL_{y})} \int_{\LL_{y}} f \circ \pi \, d\xmeasure  \right|
   \leq \omega_f(\tau_{\BB}) + 2\|f\|_{\infty}\delta_{\BB}. 
\end{equation*}

Let $y_0 \in U_{\y}$ and let $\lambda$ be a  measure on $U_{\y}$ as in
equation~\eqref{integral}. Notice that $\nu(B) = \int_{U_{\y}} \xmeasure(\LL_y)\ d\lambda(y)$. Therefore
\begin{align*}
    & \left| \frac{1}{\nu(\B)} \int_{\B} f \ d\nu -
      \frac{1}{\xmeasure(\LL_{y_0})} \int_{\LL_{y_0}} f \circ \pi \, d\xmeasure \right| \\ 
    \leq & \left| \frac{1}{\nu(\B)} \int_{U_{\y}} \left( \int_{L_{y}} f \circ \pi \, d\xmeasure \right) d\lambda(y) - \frac{1}{\xmeasure(\LL_{y_0})} \int_{\LL_{y_0}} f \circ \pi \, d\nu_{ \xform} \right|  \\ 
    \leq & \frac{1}{\nu(B)} \int_{U_{\y}} \xmeasure(\LL_y) \left| \frac{1}{\xmeasure(\LL_{y})} \int_{\LL_{y}} f \circ \pi \, d\xmeasure  - \frac{1}{\xmeasure(\LL_{y_0})} \int_{\LL_{y_0}} f \circ \pi \, d\xmeasure \right| \ d\lambda(y) \\
    \leq & \omega_f(\tau_{\BB}) + 2\|f\|_{\infty}\delta_{\BB}.
\end{align*}
\end{proof}

\subsection{Mixing of geodesics, nondivergence of horocycles}
We recall the following useful results:

\begin{lem}[Nondivergence of the horocycle flow \cite{MW}] \label{nondivergence}
For any $\varepsilon>0$ and $c>0$ there is a compact $K \subset \M^{(1)}$ such that for any $q \in \M^{(1)}$, one of the following holds: 

\begin{itemize}
    \item $\underset{T \to \infty}{\mathrm{liminf}} \ \frac{1 }{T} \int_0^T \mathbf{1}_K(u_s  q) \, ds> 1 - \varepsilon$ (where $\mathbf{1}_K$ is the indicator of $K$). 
    \item The surface $q$ has a horizontal saddle connection of length smaller than $c$. 
\end{itemize}
   
\end{lem}

For any $0 < c \leq \infty$, let $\M_{<c}$ be the subset the subset of $\M^{(1)}$ consisting of surfaces which have a horizontal saddle connection of length smaller than $c$, and let  $\M_{\geq c} = \M^{(1)} - \M_{<c}$. We deduce the following corollary.
\begin{lem} \label{nondivergence cor}
For any $0 < \varepsilon < 1 $ and $0 < c \leq \infty$, there is a compact $K \subset \M^{(1)}$ such that for any $U$-invariant measure $\mu$ on $\M^{(1)}$, $$\mu(K) > (1-\varepsilon)\mu(\M_{\geq c}).$$
\end{lem}
\begin{proof}
Given $\varepsilon$ and $c$, let $K$ be a compact set given by Lemma \ref{nondivergence} (if $c= \infty$, we can apply Lemma \ref{nondivergence} to any finite $c$). An application of a generalisation  of the Birkhoff ergodic theorem for locally finite measures (see \cite[Thm. 2.3]{Krengel} for a general formulation) to the invariant measure $\mu$ and the function $\mathbf{1}_K$ shows that there is a non-negative function $f \in L^1(\mu)$ such that $\|f\|_{L^1(\mu)}  \leq \|\mathbf{1}_{K}\|_{L^1(\mu)} = \mu(K)$ and for $\mu$-almost every $q \in \M^{(1)}$, 
\begin{equation*}
    \frac{1}{T} \left| \{ s \in [0,T] : \ u_s  q \in K\} \right| \underset{T \to \infty}{\longrightarrow} f(q). 
\end{equation*}
By Lemma \ref{nondivergence},  we have that for almost every $q \in \M_{\geq c}$, $f(q)>1 - \varepsilon$. As a consequence,
$$
\mu(K) \geq \int_{\M} f \, d\mu >(1-\varepsilon)\mu(\M_{\geq c}).
$$ 
\end{proof}

Lemma \ref{nondivergence cor} will be used at several places in this text. The first fact we deduce from it is the following:  

\begin{lem}\label{proportion} Let $\nu$ be a saddle connection free horospherical measure and let $\delta>0$. Then there is a regular box $\BB \subset \pi^{-1}(\M)$, a constant $c>0$ and an unbounded increasing sequence of times $t_i$ such that:

\begin{enumerate}[(a)]
    \item \label{item: lemma 1} For all $i \geq 0$, $\nu(\B_{t_i}) > c\nu (\M)$, where $\B_{t_i} = \pi(\BB_{t_i})$.
    \item \label{item: lemma 2} Both the thickness and distortion of each $\BB_{t_i}$ are smaller than $\delta$. 
  \end{enumerate}
In particular it follows from (a) that $\nu$ is finite. 

\end{lem}

\begin{proof}
Let $K$ be a compact subset as in Lemma \ref{nondivergence cor} for $\varepsilon = \frac12, c=\infty$, and denote $\nu_t \df (g_{-t})_{\ast}\nu$. By Proposition \ref{prop: horocycle invariance}, $\nu$ is $U$-invariant, and since $g_t$ normalizes $U$, the same holds for $\nu_t$. Since $\nu$ is saddle-connection free, so is $\nu_t$. So, applying Lemma \ref{nondivergence cor} to $\nu_t$,
$$\nu_t(K) > \frac{\nu_t(\M_{\geq \infty})}{2} = \frac{\nu_t(\M)}{2}.
$$



For every $\delta>0$, using Proposition \ref{prop: pushofbox2}, $K \cap \M^{(1)}$ can be covered by the image by $\pi$ of regular boxes $\BB_1, \ldots, \BB_N$ whose distortion is smaller than $\delta$, and for which the thickness of $g_t(\BB_j)$ is smaller than $\delta$, for each $j$ and each $t \geq 0$. By Lemma \ref{prop: pushofbox1}, the distortion of $g_t(\BB_j)$ is also less than $\delta$ for each $j$ and each $t \geq 0$. Let $c \df \frac{1}{2N}$. For each $t$, there is $j=j(t) \in \{1, \ldots, N\}$ such that
$$\nu_t\left(B_j \right) \geq \frac{\nu_t(K) }{N}> c \, \nu_t(\M).
$$
Let $t_i
\to \infty$ be a sequence along which $j = j(t_i)$ is constant. Then \ref{item:
  lemma 1} and \ref{item: lemma 2} hold  for $\BB = \BB_j$. 
\end{proof}

\begin{lem}[Mixing of the geodesic flow] \label{mixing} 
For any invariant subvariety $\M$, the geodesic flow is mixing with
respect to the special flat measure on 
$\M^{(1)}$. 
\end{lem}

For a proof and detailed discussion of this result and its
quantitative strengthenings, see
\cite[Chap. 4]{Forni-Matheus_survey} or \cite{EMM_JEMS}.

\subsection{Putting it all together}
We have gathered all the ingredients needed to give the proof of one of our main results.  

\begin{proof}[Proof of Theorem \ref{thm: classification}]
Let $\nu$ be a saddle connection free horospherical measure. We assume first that $\nu$ is ergodic for the horospherical foliation. We will show that the special flat measure $m_{\M}$ is absolutely continuous with respect to $\nu$. To see this, let $A$ be a Borel set of positive measure for $m_{\M}$. Since $m_{\M}$ is a Radon measure, in particular inner regular, there is a compact $K$ contained in $A$ such that $m_{\M}(K) > 0$. Let $U$ be an open set that contains $A$ and let  $f:\M^{(1)} \to [0,1]$ be a continuous function whose support is contained in $U$ and that evaluates to $1$ on $K$. Such a function exists by Urysohn's Lemma. Let $\varepsilon >0$, and choose $\delta>0$ so that

$$
\omega_f(\delta) + 2\|f\|_\infty \delta < \vre.
$$

By Lemma \ref{proportion}, there is $c>0$,  a regular box $\BB$ and $t_i \to \infty$ such that
for each $i$, $\tau_{\BB_{t_i}} < \delta$ and $\delta_{\BB} < \delta, $ and $\nu(B_{t_i}) \geq c \nu (\M).$ Applying Lemma
\ref{comparison} to both  $\nu$ and $m_{\M}$ we obtain

$$
\left| \frac{1}{\nu(B_{t_i})} \int_{\B_{t_i}} f \ d\nu -\frac{1}{m_{\M} (\B_{t_i})} \int_{B_{t_i}} f \ dm_{\M} \right| <2\varepsilon. 
$$

By mixing of the geodesic flow with respect to $m_{\M}$, there is $i>0$ large enough such that $m_{\M}(B_{t_i} \cap K) > m_{\M}(B) (m_{\M}(K) - \varepsilon)$. Therefore: 

\[\begin{split}
    \frac{\nu(U)}{c \, \nu(\M)}   \geq &  \frac{\nu(U)}{\nu(\B_{t_i})} \geq \frac{1}{\nu(\B_{t_i})}\int_{\B_{t_i}} f \, d\nu \\  
     > & \frac{1}{m_{\M}(\B_{t_i})}\int_{\B_{t_i}} f \, dm_{\M} - 2\varepsilon  \\
     \geq  & \frac{m_{\M}(\B_{t_i} \cap K)} {m_{\M}(\B_{t_i})} -2\varepsilon > m_{\M}(K)-3\varepsilon.  
\end{split}\] 
Since $\varepsilon$ was chosen arbitrarily, we have proven  $\nu(U)
\geq c \, \nu (\M) \, m_{\M}(K)$. Since this holds for an arbitrary open $U$ containing $A$, and $\nu(\M)$ is finite, we deduce by outer regularity of the
measure $\nu$ that $\nu(A)$ is positive. This completes the proof that $m_{\M} \ll \nu$.

It follows from Proposition \ref{prop: ergodic decomposition} that $m_{\M} = c\nu$ for some $c \geq 0$, and since $m_{\M}$ is nonzero, $c>0$ and $\nu = \frac{1}{c} m_{\M}$.  For general $\nu$, we obtain from the case just discussed  that all the ergodic components of the measure $\nu$ are proportional to the special flat measure and thus $\nu$ itself is proportional to the special flat measure. 
\end{proof}

\section{Examples of horospherical measures}\label{sec: examples}
The simplest example of a horospherical measure which is not the special flat measure occurs when $\M$ is a closed $\GL$-orbit. In this case the leaves of the horospherical foliation are the $U$-orbits, and the length measure on a closed periodic $U$-orbit is a horospherical measure; indeed, in this case, the transverse measure $\lambda$ in equation~\eqref{integral} is atomic.

In order to obtain more complicated examples, we use the following:

\begin{prop}\label{prop: a measure on a closed horosphere}
Let $W^{uu}(q)$ be a closed horosphere in $\M$. Then $W^{uu}(q)$ is the support of a horospherical measure $\nu$ whose lifts are the measures $\nu_{\xform,\q}^L$ where $L$ is a lift of $\M$ and $\pi(\q) = q$.   
\end{prop}
  
\begin{proof}
The horosphere $W^{uu}(q)$ is closed if and only the collection $\bigl\{ \WW^{uu}(\q) \bigr\}$ is locally finite, where $\q$ ranges over $\pi^{-1}(q)$ and $L$ ranges over the lifts of $\M$ that contain $\q$. Each of the $\WW^{uu}(\q)$ carries the Radon measure $\nu_{\xform,\q}^L$ and the measure 
$$
\tilde{\nu} \df \sum \nu_{\xform,\q}^L
$$
is a $\Mod(S,\Sigma)$-invariant Radon measure on $\pi^{-1}(\M)$. Let $\nu$ be the Radon measure on $\M$ whose lift is $ \tilde{\nu}$ (see Proposition \ref{prop: correspondence}). The measure $\nu$ is horospherical by construction. 
\end{proof}

To construct an example of a closed horosphere, we use {\em horizontally periodic} surfaces, i.e., surfaces which can be represented as a finite union of horizontal cylinders. Let $\M = \mathcal{H}(1,1)$. This stratum is an invariant subvariety of dimension 5 (see Definition \ref{defin: invariant subvariety}), and thus its horospherical leaves have real dimension 4. Let $a,b$ be real numbers with $a, b  \in (0,1) $ and $0 < b < \min(a, 1-a) $, let $\tau_1, \tau_2 \in \mathbb{S}^1 \df \R /\mathbb{Z}$,  and set $\bar \tau_1  \df a\tau_1$ and $\bar \tau_2 \df (1-a)\tau_2$, so that $\bar \tau_1, \bar \tau_2$ take values in circles of circumference $a, 1-a$ respectively. Define the surface $q=q_{a,b, \tau_1, \tau_2} \in \M$ by the polygonal representation shown in Figure \ref{fig: figure 1}. In the horizontal direction it is comprised of two cylinders, each of height $0.5$, and of areas $0.5a$ and $0.5(1-a)$. The parameters $\tau_1, \tau_2$ are called {\em twist parameters}. Changing them by adding an integer amounts to performing the corresponding number of Dehn twists in the two cylinders, and thus does not change the surface $q$.

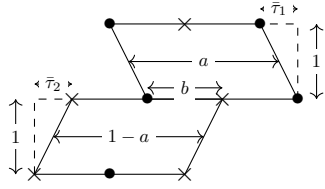
\begin{figure}
\begin{center}
\begin{tikzpicture}

\draw (0,0) -- (2,0) -- (2.5,1) -- (.5,1) -- cycle;
\draw (1.5,1) -- (1,2) -- (3,2) -- (3.5,1) -- cycle;

\draw [dashed] (0,0) -- (0,1) -- (.5,1) ;
\draw [dashed] (3.5,1) -- (3.5,2) -- (3,2);

\draw[<->] (0,1.2) -- (0.5,1.2) node[midway, fill=white, scale=0.7]
{$\bar \tau_2$};
\draw[<->] (3,2.2) -- (3.5,2.2) node[midway, fill=white, scale=0.7]
{$\bar \tau_1$};
\draw[<->] (1.5,1.15) -- (2.5,1.15) node[midway, fill=white, scale=0.7] {$b$};
\draw[<->] (.25,.5) -- (2.25,.5) node[midway, fill=white, scale=0.7] {$1-a$};
\draw[<->] (1.25,1.5) -- (3.25,1.5) node[midway, fill=white, scale=0.7] {$a$};
\draw[<->] (-0.25,0) -- (-.25,1) node[midway, fill=white, scale=0.7] {$1$};
\draw[<->] (3.75,1) -- (3.75,2) node[midway, fill=white, scale=0.7] {$1$};

\draw (0,0) node{$\times$};
\draw (2,0) node{$\times$};
\draw (2.5,1) node{$\times$};
\draw (.5,1) node{$\times$};
\draw (2,2) node{$\times$};

\draw (1.5,1) node{$\bullet$};
\draw (1,2) node{$\bullet$};
\draw (3,2) node{$\bullet$};
\draw (3.5,1) node{$\bullet$};
\draw (1,0) node{$\bullet$};

\end{tikzpicture}
\end{center}
\caption{A completely periodic surface in $\HH(1,1)$. The two singularities are marked with $ \bullet$ and $\times$.  } \label{fig: figure 1}
\end{figure}

It is clear that varying  the parameters $a, b, \tau_1, \tau_2$ results in surfaces that belong to the horospherical leaf of $q$, and thus, by a dimension count, they locally parameterize the leaf of $q$. In either of the cases $ b\to 0$ or $b \to \min(a, 1-a)$, the surfaces $q_{a,b, \tau_1, \tau_2}$ have shorter and shorter horizontal saddle connections on the boundaries of the cylinders, and thus exit compact subsets of $\M^{(1)}$. This means that the horosphere $W^{uu}(q)$ is closed and that the map 
$$
\mathbb{S}^1 \times \mathbb{S}^1 \times \{(a,b) \in (0,1)^2  :  0<  b < \min(a, 1-a)\} \to \M^{(1)}\ \ \ (a,b,\tau_1, \tau_2) \mapsto q_{a,b,\tau_1, \tau_2} 
$$ 
is a proper embedding whose image is $W^{uu}(q)$.

It can be checked that in this case the map $(a,b, \bar \tau_1, \bar \tau_2)
\mapsto \dev(q_{a,b,\tau_1, \tau_2} )$ is affine in charts. Thus the horospherical measure can be written explicitly (up to scaling) as  $d\nu(q_{a,b,\tau_1, \tau_2}) =  da \, db \, d\bar{\tau}_1 \, d\bar{\tau}_2$. 

\begin{rem}
For the horospherical measure constructed in the preceding example, the space $Z_{\M}$ (defined in equation~\eqref{eq: def ZM}) is one dimensional, and for every surface $q$ in  the support of this measure, $Z^{(q)}$ (defined in equation~\eqref{eq: set where defined}) is a bounded interval. Moreover, for any $v \in Z_{\M}$ there is a positive measure set of surfaces $q$ (with small values of $a$) for which $\Rel_v(q)$ is not defined. This shows that the hypothesis in Proposition \ref{prop: real REL invariance}, that $\Rel_v(q)$ is defined, is not always satisfied. (More explicitly, a leaf of the real REL foliation is given by varying $b$.)
\end{rem}

It is no coincidence that the closed horospheres in the two preceding examples consist of horizontally periodic surfaces. 

\begin{prop}\label{prop: properly embedded}
For any $\M$ and any $q \in \M$, the surface $M_q$ is horizontally periodic if and only if  $W^{uu}(q)$ is closed. In this case every surface in $W^{uu}(q)$ is horizontally periodic, and the horospherical measure on $W^{uu}(q)$ constructed in Proposition \ref{prop: a measure on a closed horosphere} is finite. 
\end{prop}

\begin{proof}
Suppose first that $M_q$ is horizontally periodic, and let $f: S \to M_q$ be a marking map representing $\q \in \pi^{-1}(q)$. Let $C_1, \ldots, C_s$ be the horizontal cylinders on $M_{q}$, and let $c_j, h_j$ denote respectively the circumference and height of $C_j$. Since the area of $M_q$ is one,
\begin{equation}\label{eq: area is one}
\sum_{j=1}^s  c_j h_j =1. 
\end{equation}

Let $\alpha_1, \ldots, \alpha_r, \alpha_{r+1}, \ldots, \alpha_{r+s}$ be a collection of oriented paths in $S$ with endpoints in $\Sigma$ which satisfy the following:
\begin{itemize}
    \item The collection $\{f(\alpha_i): i =1, \ldots, r+s\}$ consists of saddle connections.
    \item The collection $\{f(\alpha_i): i = 1, \ldots, r\}$, is the set of all the horizontal saddle connections on cylinder boundaries, and these are oriented so that the horizontal coordinate increases.
    \item For $j = 1, \ldots, s$, the saddle connection $f(\alpha_{r+j})$ is contained in $C_j$, crosses $C_j$, and is oriented so that the vertical coordinate increases. 
\end{itemize}

If $C$ is a cylinder and $\sigma$ is a saddle connection on a translation surface, we say that {\em $\sigma$ crosses $C$} if it intersects all the core curves of $C$.

These paths represent classes in $H_1(S, \Sigma)$ and they give a generating set for $H_1(S, \Sigma)$.
Write the holonomies $\hol(M_{\q}, \alpha_i)$ as 

\begin{equation}\label{eq: bounds cylinders}
  \begin{split}
\hol(M_{\q}, \alpha_i) = & (t_i, 0) \ \ \ i =1, \ldots, r \\
\hol(M_{\q}, \alpha_{r+j}) = & (\tau_j, h_j) \ \ \ j=1, \ldots, s.
    \end{split}
  \end{equation}
  For each $j$ and each boundary component of $C_j$, we
  have
  \begin{equation}\label{eq: on boundary}
    \sum_{i \in \mathcal{I}} t_i = c_j,
  \end{equation}
  where $\mathcal{I}$ is
  a subset of $\{1, \ldots, r\}$ containing the saddle connections
  comprising the boundary component. The numbers $t_i, h_j, \tau_j$
  also satisfy some linear equations $L_1, \ldots, L_t$, which
  describe the space that $L$ is modeled on, in a neighborhood of $\q$. 

Let $\q' \in \WW^{uu}(\q)$. We first show that the underlying surface
$M_{q'}$ is horizontally periodic. There is a continuous path $\sigma
\mapsto \q(\sigma)$, 
with $\q(0) = \q$ and $\q(1)=\q'$, such that $\q(\sigma) \in \WW^{uu}(\q)$ for
every $\sigma \in [0,1]$. 
By definition of the horospherical foliation, for any $\sigma \in [0,1]$,
$\dev(\q(\sigma)) - \dev(\q) \in H^1(S, \Sigma; \R_{\x})$. 
That is to say, there are $t_1(\sigma), \ldots, t_r(\sigma), \tau_1(\sigma),
\ldots, \tau_s(\sigma) \in \R$ such 
that equation~\eqref{eq: bounds cylinders} holds for $\q(\sigma)$. Note that
$h_j$ is independent of $\sigma$, that the numbers $t_i(\sigma),
h_j, \tau_j(\sigma) $ also
satisfy the equations $L_1, \ldots, L_t$, and that the numbers
$c_j(\sigma)$ defined by equation~\eqref{eq: on boundary} also satisfy
equation~\eqref{eq: area is one}.

Assume first that
\begin{equation}\label{eq: positivity holds}
t_i(\sigma)>0 \  \ \ \text{ for all } \sigma \in [0,1] \ \text{ and }
i \in \{1, \ldots, r\}.
\end{equation}
For each $j \in \{1, \ldots, s\}$, the
set
$$\{\sigma \in [0,1] : \text{ the curve  } \alpha_{r+j} \text{
  crosses a horizontal cylinder on } \q(\sigma) \}$$
is open (this is a general property of cylinders, see e.g. \cite[\S
4.1]{MT}) and 
closed (since the heights $h_j$ are fixed). Therefore, by a connectedness
argument, $\q'$ is also made of $s$ horizontal cylinders. By 
equation~\eqref{eq: area is one}, these 
cylinders give a set of full measure in $\q(\sigma)$, and thus $\q(\sigma)$ is
horizontally periodic. 

Now if equation~\eqref{eq: positivity holds} fails, let $\sigma_{\min}$ be the
smallest value of $\sigma$ for which it fails. When $\sigma$ increases to
$\sigma_{\min}$ from below, the surfaces $\q(\sigma)$ have shorter and shorter
horizontal saddle connections on the boundaries of cylinders, and this
means that the surfaces $\q(\sigma)$ cannot converge to $\q
(\sigma_{\min})$. This shows that equation~\eqref{eq: positivity holds} holds and
proves that all surfaces in $\WW^{uu}(\q)$ are horizontally periodic. 
  \ignore{
As
long as there are numbers $c'_i$ satisfying equation~\eqref{eq: area is one},
and numbers $t'_i>0$ satisfying the equations \eqref{eq: on boundary} (with
$c_i, t_i$ replaced by $c'_i, t'_i$), there is a surface $M_{\q'}$ made
of $s$ cylinders with circumferences $c'_j$, heights $h_j$, glued in
the same way that they are glued on $M_{\q}$. Clearly we then have $\q'
\in \WW^{uu}(\q)$.

\textcolor{red}{Florent asks whether injectivity of the developing map
  is implicitly being used here. Also there are more equations coming
  from $\M$ that are missing.}
}

The set of parameters $t_i, \tau_j$ giving surfaces in $W^{uu}(q)$ is bounded. Indeed, the $c_i$ defined by equation~\eqref{eq: on boundary} are bounded by equation~\eqref{eq: area is one},  and this implies that the numbers $t_i \in (0, \max_j c_j)$ are bounded. Changing the $\tau_j$ by adding an integer multiple of $c_j$ amounts to performing Dehn twists in the cylinder $C_j$ and does not change the projection of the surface to $\M$. That is, the numbers $\tau_j$  can be taken to lie in the bounded set $[0, c_j )$. Also, as the parameters $t_i$ leave compact subsets of the bounded domain described above, at least one of the horizontal saddle connections on the corresponding surface has length going to zero. This implies that the bounded set of surfaces we have just described by varying the parameters $t_i, \tau_j$ projects to the entire leaf $W^{uu}(q)$, that this leaf is properly embedded, and that all surfaces in this leaf are horizontally periodic.

Furthermore, we can use equation~\eqref{eq: area is one} to express $c_1$ as a function of $c_2, \ldots, c_s$ (a constant function when $s=1$), and using the linear equations defining $L$, we can write some of the variables $c_j, \tau_j, t_i$ as linear combinations of a linearly independent set of variables. We can then write the horospherical measure up to scaling as $d\nu(q) = \prod_{j=\mathcal{J}_1} dc_j  \, \prod_{j\in \mathcal{J}_2} d\tau_j \, \prod_{i \in \mathcal{J}_3} dt_i $, for some subsets of indices, and thus the preceding discussion shows that the total measure of the leaf is bounded.   

Now suppose  that $M_q$ is not horizontally periodic. According to \cite{SW_calanque}, the horocycle orbit $Uq$ consists of surfaces that are not horizontally periodic, but there is $q' \in \overline{Uq}$ such that $M_{q'}$ is horizontally periodic. By Proposition \ref{prop: horocycle invariance}, $Uq \subset W^{uu}(q)$, and thus $q' \in \overline{W^{uu}(q)}$. Since $M_q$ is not horizontally periodic, according to the first part of the proof, $q' \notin W^{uu}(q)$. This shows that the leaf $W^{uu}(q)$ has an accumulation point that is not contained in the leaf, which is to say that $W^{uu}(q)$ is not closed. 
\end{proof}

\subsection{Classification of horospherical measures in the eigenform loci in
  $\HH(1,1)$}\label{subsec: eigenform loci classification}
The stratum $\HH(1,1)$ contains a countable collection of complex $3$-dimensional invariant
subvarieties known as {\em eigenform loci}. This terminology is due to
McMullen, who gave a complete classification of these invariant
subvarieties in a sequence of papers (see \cite{McMullen_SL2} and references
therein), following the first such examples 
discovered by Calta \cite{Calta}. The horocycle invariant measures and
orbit-closures for the $U$-action on an eigenform locus, were
classified in \cite{BSW} (these classification results require Theorem
\ref{thm: classification} of the present work). We can classify the
horospherical measures inside eigenform loci as follows:

\begin{thm}\label{thm: eigenform locus}
Let $\M$ be an eigenform locus in $\HH(1,1)$, and let $\nu$ be an ergodic horospherical measure on $\M$. Then either $\nu$ is the special flat measure $m_{\M}$ or $\nu$ is the measure given by Proposition \ref{prop: a measure on a closed horosphere} on a closed horosphere $W^{uu}(q)$ of a horizontally periodic surface $q \in \M$. 
\end{thm}

\begin{proof}

Suppose that $\nu$ is neither the special flat measure nor the measure supported on a closed horosphere of a horizontally periodic surface. This immediately rules out cases (1), (2), and (7) of the classification of $U$-ergodic measures in \cite[Thm. 9.1]{BSW}. By Theorem \ref{thm: classification}, $\nu$ cannot be saddle-connection free, which rules out case (5). In each of the three remaining cases (3), (4), and (6), $\nu$-a.e. surface has exactly one horizontal saddle connection or exactly two homologous horizontal saddle connections forming a horizontal slit. We conclude that in each of these three cases, $\nu(\M_{\geq \infty}) = 0$, and moreover that for $\nu$-a.e. surface we can lengthen or shorten all horizontal saddle connections by moving in the real REL leaf. It follows that $\mathrm{Rel}_s(\M_{\geq c})$ and $\M_{\geq c+s}$ differ on a set of $\nu$ measure zero. Since $\nu$ is REL-invariant, $\nu(\M_{\geq c}) = \nu(\mathrm{Rel}_s(\M_{\geq c}))$, and we conclude that the quantity $\nu(\M_{\geq c})$ does not depend on $c$ for any finite $c$. 

By Lemma \ref{nondivergence cor} applied to any positive $c$, this quantity is bounded by $\nu(K)$ for some compact set $K$, and is therefore finite. So, taking the limit as $c \to 0$, we see that $\nu(\M) = \nu(\M_{\geq c})$ for any $c$, and then taking the limit as $c \to \infty$, we conclude that $\nu(\M) = \nu(\M_{\geq \infty})$, which is equal to 0 from above. This absurdity rules out the remaining cases (3), (4), and (6).

      \end{proof}

\subsection{An example of horospherical measure in $\HH(2)$}\label{subsec: H(1,1)
  classification}
  Since there is currently no classification of horospherical measures in $\HH(2)$,
  it is of interest to give examples.
In this subsection we construct an ergodic horospherical measure which is not
the special flat measure and is not supported on one properly embedded
horospherical leaf. Its support is contained  in the four-dimensional
invariant subvariety $\M = \HH(2)$, the genus two
stratum consisting of surfaces with one singular point of order two. 

Recall from Corollary \ref{cor: equivalent diagrams} that for a given
ergodic horospherical measure, almost all surfaces are horizontally
equivalent. In Figure \ref{fig: in H(2)} we
show a typical surface $q$ for our horospherical measure, and a typical
topological picture of its horizontal
saddle connections. These saddle connections will be denoted by
$\delta$ and $\delta'$. They disconnect the surface into
a horizontal cylinder $C$, shaded gray in Figure \ref{fig: in H(2)},
and a torus $T$.

\begin{figure}[h]
  \begin{tikzpicture}
\begin{scope}[xshift=-1cm,yshift=1cm] 
\coordinate (A) at (0,0);
\coordinate (B) at (0,-1.3028);
\coordinate (C) at (1.3028,-1.3028);
\coordinate (D) at (1.3028,0);
\coordinate (E) at (1.8,-0.41);
\coordinate (F) at (2.8,0.59);
\coordinate (G) at (2.3028,1);
\coordinate (H) at (1,1);
\coordinate (GE) at (2.1514,0);
\coordinate (EG) at (3.1514,1);
\coordinate (G1) at (2.3028,0.);
\coordinate (G2) at (3.3028,1);
\coordinate (G3) at (3.3028,0);
\coordinate (G4) at (0.3028,1);
\coordinate (G5) at (0.3028,0);
\fill [gray!10] (A) -- (B) -- (C) -- (D) -- cycle; 
\draw (A) -- (D);
\draw (A) node {\textbullet}
--
(B) node {\textbullet}
--
(C) node {\textbullet}
--
(D) node {\textbullet}
--
(E) node {\textbullet}
--
(F) node {\textbullet}
--
(G) node {\textbullet}
--
(H) node {\textbullet}
--
(0,0) -- cycle;
\draw[->, >=stealth, thick] (A)--(D) ;
\draw[->, >=stealth, thick] (B)--(C) ;
\draw[->, >=stealth, thick] (H)--(G) ;
\draw[->, >=stealth, thick, densely dotted]  (D)--($(D)+(0.7,0)$);
\draw[->, >=stealth, thick, densely dotted]  ($(F)-(0.7,0) $)--(F);
\end{scope}

\begin{scope}[xshift = 6cm]
\coordinate (A) at (1,1);
\coordinate (B) at (-1,1);
 \node (black) at (0,1) [circle,draw,fill=black,inner sep=0pt,minimum size=1.5mm] {};
\draw[->,>=stealth, thick] (black) to [out=325,in=270,looseness=1.5] (A);
\draw[->,>=stealth, thick] (B) to [out=90,in=145,looseness=1.5] (black);
\draw[->, >=stealth, thick] (A) to [out=90,in=35,looseness=1.5] (black);
\draw[->, >=stealth, thick] (black) to [out=205,in=270,looseness=1.5] (B);
\draw[->, >=stealth, thick, densely dotted]  ($ (black)-(0,.9) $)--(black);
\draw[->, >=stealth, thick, densely dotted]   (black)--($ (black)-(0,
-.9) $);
\end{scope}
\end{tikzpicture}
\caption{A surface in $\HH(2)$ with two horizontal saddle connections,
  bounding a horizontal cylinder. On the right, the corresponding horizontal saddle connection diagram.}\label{fig: in H(2)}
\end{figure}
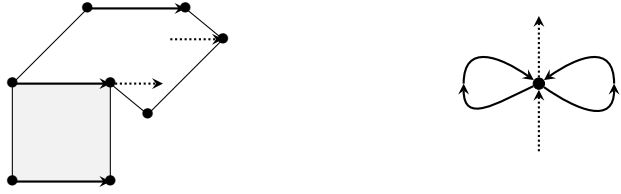

Let $x$ be the length of $\delta$ and $\delta'$, let $\eta$ be a saddle
connection passing from top to bottom of the cylinder $C$, and let its
holonomy be $(a, \tau)$. Fix $\q \in \pi^{-1}(q)$. The height of $C$
is constant and equal to $a$ in a neighborhood of $\q$ in
$\WW^{uu}(\q)$. The area of $C$ is $ax$, and hence
\begin{equation}\label{eq: bedD}
0< x < \frac{1}{a}.
\end{equation}
Moreover, changing $\tau$ by an integer multiple of $x$ amounts to
performing a Dehn twist in $C$ so does not change the surface
$M_q$. Thus we 
may take 
\begin{equation}\label{eq: bedD2}
\tau \in [0, x).
\end{equation}
When varying surfaces within their horospherical leaves, we change
horizontal components of all saddle connections, and thus changing
$\tau$ and $x$ we stay in the horospherical leaf. Similarly, by
Proposition \ref{prop: horocycle invariance}, $u_{s'}q \in W^{uu}(q)$ for
every $s'$.  Moreover, if $M_q=C \cup T$ as
above, the surface $u^{(T)}_s M_q$ obtained by performing the
horocycle flow on $T$ and leaving $C$ unchanged is also in
$W^{uu}(q)$. 
It is easy
to check that changing the three parameters $x,\tau, s$ gives a linear
mapping in period coordinates, and that the three corresponding
tangent directions in  directions in $T_q(\M)$ are linearly independent. Since
$\dim (\M)=4$, the dimension of the horospherical leaves in $\M$ is
three, so the variables $x, \tau, s$ give an affine parameterization of a neighborhood of
$q$ in $W^{uu}(q)$.

Since the height $a$ of $C$ remains constant in $W^{uu}(q)$, by equations~\eqref{eq:
  bedD} and \eqref{eq: bedD2}, the variables $x,\tau$ take values in the
bounded domain
$$
\Delta \df \left\{(x,\tau): 0 \leq  \tau < x < \frac{1}{a} \right\}.
$$
We construct a bundle $\mathcal{B}$ with base $\Delta$, and a homogeneous space
fiber, as follows.
Let  $\mathrm{Tor} \df G/\mathrm{SL}_2(\mathbb{Z})$, the space of tori of
some fixed area. This area is usually taken to be one, but by
rescaling, can be taken to be any fixed number. For each $x \in
\left(0, \frac{1}{a} \right)$, let 
$\mathrm{Tor}(x)$ denote the space of tori of area $1-ax$ and with an
embedded horizontal segment of length $x$. This is the complement in
$\mathrm{Tor}$ of a closed set with empty interior (consisting of
periodic horocycles of period at most $x$). Define $\mathcal{B}$ to be 
the bundle with base $\Delta$ and such that the fiber 
over $(x, \tau) \in \Delta$ is $\mathrm{Tor}(x)$.

Let  
$\mu$ be the $G$-invariant probability measure on
$\mathrm{Tor}$. Since the set of surfaces which do not admit an
embedded horizontal segment of some length is of $\mu$-measure zero,
we can also think of $\mu$ as a probability measure $\mu_x$ on
$\mathrm{Tor}(x)$. 
For  $(x, \tau) \in \Delta$ let $C = C(x,\tau)$ be a
cylinder  of height $a$, circumference $x$ and
twist $\tau$. We have a map 
$$\Psi: \mathcal{B} \to 
\mathcal{H}(2)$$
defined by 
gluing the torus $T$ from $\mathrm{Tor}(x)$, with a slit of length $x$, to the
cylinder $C(x, \tau)$. Let
$$\nu \df \int_0^{1/a} \int_0^x
\Psi_*(\mu_x) \, d\tau \, dx.$$
The image $\Psi(\mathcal{B})$ is a five-dimensional properly embedded
submanifold of $\M$, consisting of all surfaces that can be presented 
as in Figure \ref{fig: in H(2)} for some fixed choice of $a>0$. Along
any sequence of elements  
$(x,\tau) \in \Delta$ leaving compact subsets, we have either $x \to 0$
or the area $1-ax$ of $T$ goes to zero, and in both cases the surfaces in
the image of $\Psi$ have short saddle connections. This shows that
$\Psi(\mathcal{B})$ is properly embedded. Since $\nu$ is
invariant under translations 
using the affine coordinates $x,\tau,s$, it is a finite
horospherical measure supported on $\Psi(\mathcal{B})$.

\section{The geodesic flow and weak unstable foliation}\label{sec: geo and fol}

\begin{proof}[Proof of Theorem \ref{thm: Forni question}]
Let $\mu$ be a finite horospherical measure, and let $\mu_t \df g_{t*}\mu$. Our goal is to show that $\mu_t \to_{t \to \infty} m_{\M}$.
In order to prove that $\mu_t \to m_{\M}$, it is enough to show that in any subsequence $t_n \to \infty$ one can find a further subsequence $t'_n$ so that $\mu_{t'_n} \to m_{\M}$. This will be accomplished in two steps. In the first step we will pass to a subsequence along which $\mu_{t'_n} \to \mu_\infty,$ and show that $\mu_\infty$ is also a probability measure. In the second step we show that $\mu_\infty$ is saddle connection free. Since $\mu_\infty$ is also horospherical by item (1) of Proposition \ref{prop: ergodic decomposition}, an application of Theorem \ref{thm: classification} then completes the proof.

Since $\mu(\M)$ is finite, we can renormalize so that $\mu(\M) = 1$. For the first step, we need to show that the sequence of measures $\{\mu_{t_n}\}$ is {\em tight}, i.e., for any $\vre>0$ there is a compact $K \subset \M$ such that for all large enough $n$, $\mu_{t_n}(K) \geq 1-\vre.$ For this we will use Lemma \ref{nondivergence cor}.

Since $\mu(\M) = 1$, there is a $c$ small enough that $\mu(\M_{\geq c}) > 1 - \frac\vre2$. By Lemma \ref{nondivergence cor} there is a compact $K \subset \M$ such that 
$$
\nu(K) > \left(1-\frac\vre2 \right)\nu(\M_{\geq c})
$$
for ever $U$-invariant measure $\nu$. Applying this to $\nu = \mu_t$ for any $t \geq 0$ gives
\begin{align*}
\mu_t(K) &> \left(1-\frac\vre2 \right)\mu_t(\M_{\geq c})\\
&= \left(1 - \frac{\vre}{2} \right)\mu(\M_{\geq e^{-t}c})\\
&> \left(1-\frac{\vre}{2} \right)^2\\
&> 1 - \vre,
\end{align*}
where the penultimate inequality uses $\M_{\geq c} \subset \M_{\geq e^{-t}c}$.

By tightness, there exists a subsequential limit that is a probability measure. Now, letting $\mu_\infty$ be any limit along a subsequence $t_n$, it remains to show that $\mu_\infty$ is saddle-connection free. We will show that for any $\vre > 0$ and any $C < \infty$, $\mu_\infty(\M_{<C}) < \vre$. Choose $c$ small enough that 
$$
\mu(\M_{<c}) < \frac\vre2.
$$
Next, choose $n$ large enough so that $e^{-t_n} C < c$ and 
$$|\mu_{t_n}(\M_{<C}) - \mu_\infty(\M_{<C})| < \frac\vre2.$$
Then
$$
\mu_\infty(\M_{<C}) < \mu_{t_n}(\M_{<C}) + \frac\vre2 = \mu(\M_{<e^{-t_n} C}) + \frac\vre2 < \vre.
$$
where the last inequality uses $\M_{<e^{-t_n}C} \subset \M_{<c}$. Since $\vre$ and $C$ were arbitrary, we conclude that $\mu_\infty$ is saddle-connection free, and this concludes the proof.
\end{proof}

\begin{proof}[Proof of Theorem \ref{cor: maximal entropy}]
Let $\nu$ be a horospherical measure that is invariant by the geodesic flow. We will show that $\nu(\M_{< \infty}) = 0$. Since $g_{-t}(\M_{[c,\infty)}) = \M_{[e^{-t}c,\infty)}$ and $(g_t)_*\nu = \nu$, we see that $\nu(\M_{[c,\infty)})$ does not depend on $c$. Here the notation $\M_{[a,b)}$ means $\M_{\geq a} \cap \M_{< b}$. By Lemma \ref{nondivergence cor} applied to any particular finite $c$ and $\vre = \frac12$, there is a compact set $K$ such that
$$
\nu(\M_{[c,\infty)}) < \frac{\nu(K)}{2},
$$
and so in particular it is finite. Therefore, in the limit as $c \to \infty$, we see that $\nu(\M_{[c,\infty)}) = 0$, and then again taking the limit as $c \to 0$ we conclude $\nu(\M_{<\infty}) = 0$. Finally, by Theorem \ref{thm: classification} we conclude that $\nu$ is the special flat measure.


We now show that any leaf for the weak-unstable foliation is dense. Let $q \in \M_1^{(1)}$, let $U$ be an open set contained in $\M^{(1)}$ and let $f$ be a nonzero non negative compactly supported function whose support is contained in $U$. In order to show $U \cap W^{u}(q) \neq \emptyset$ we will show that there is $p \in W^{u}(q)$ such that $f(p)>0$. Let $\varepsilon \df \int_{\M} f \ d \mu_{\M}>0$, let $\omega_f$ denote the continuity modulus of $f$ with respect to the sup-norm distance function, and let $\q \in \pi^{-1}(q)$. Using Propositions \ref{prop: pushofbox1} and \ref{prop: pushofbox2}, let $\BB$ be a regular box containing $\q$ such that for any $t \geq 0$, the box $\BB_t \df g_t(\BB)$ satisfies $\omega_f(\tau_{\BB_t}) + 2\|f \|_{\infty} \delta_{\BB_t} < \frac{\varepsilon}{4}$. Let $m_\M$ be the special flat measure on $\M^{(1)}$. By mixing of the geodesic flow (Proposition \ref{mixing}), there is $T>0$ such that for any $t>T$, we have  

$$
\left| \frac{1}{m_{\M}(\B)} \int_{\B_t} f \, dm_{\M} - \int_{\M}f \, dm_{\M} \right| < \frac{\varepsilon}{4}. 
$$

Applying Lemma
\ref{comparison} to the special flat measure $m_{\M}$, and denoting by $\LL_t$ the plaque of $g_t \q$ in $\BB_t$, we have 

$$
\left|\frac{1}{m_{\M}(\B)}\int_{\B_t} f \, dm_{\M} - \frac{1}{\xmeasure(\LL_t)} \int_{\LL_t} f \circ \pi \ d\xmeasure \right| < \frac{\varepsilon}{4},
$$

and consequently
$$
\left|\int_{\M} f \, dm_{\M} - \frac{1}{\xmeasure(\LL_t)} \int_{\LL_t} f \circ \pi \ d\xmeasure \right | < \frac{\varepsilon}{2}.
$$

This implies $\int_{\LL_t} f \circ \pi \ d\xmeasure > 0$ and since $\pi(\LL_t)$ is
contained in $W^{u}(q)$, we obtain that there is $p \in W^{u}(q)$ such that $f(p) >0$. \end{proof}

\section{Closures of horospherical leaves}\label{sec: leaf closures}
The goal of this section is to prove Theorem \ref{thm:
  densehorosphere}. First, in order to explain the idea, we will prove
the following weaker result.

\begin{thm}\label{thm: special case first}
Let $q \in \M^{(1)}$ be a surface without horizontal saddle connections. Then $W^{uu}(q)$ is dense in $\M^{(1)}$.  
\end{thm}

\begin{proof}
Let $U$ be any open set contained in $\M^{(1)}$ and let $f$ be a nonzero non-negative function whose support is contained in  $U$.  It is enough to show that there is $p \in W^{uu}_{q}$ such that $f(p)>0$. Let $\varepsilon = \int_{\M} f \, dm_{\M} >0$, let $c = 1$,  and let $K$ be a compact subset as in Lemma \ref{nondivergence}. For any $n>0$, the surface $g_{-n}q$ does not have horizontal saddle connections and thus there is $s_n >0$ such that $p_n \df u_{s_n} g_{-n}  q $ satisfies 

\begin{equation}\label{eq: what p must satisfy}
p_n \in K \cap W^{uu}(g_{-n} q).
\end{equation}
 
The horocycle flow preserves the horospheres and the geodesic flow permutes them. As a consequence $g_n p_n \in W^{uu}(q)$. Since $K \cap \M_1^{(1)}$ can be covered by the image by $\pi$ of finitely many arbitrarily small boxes, by passing to a subsequence and using Propositions \ref{prop: pushofbox1} and \ref{prop: pushofbox2}, we can assume that there is a box $\BB \subset \pi^{-1}(\M)$ such that the translates $\BB_n = g_n(\BB)$ satisfy $\omega_f(\tau_{\BB_n}) + 2\|f\|_{\infty} \delta_{\BB_n} < \frac{\varepsilon}{4}$  and $p_n \in \pi(\BB)$, for all $n \in \mathbb{N}$. Denote by $\LL_n$ a plaque of $\BB_n = g_{n}(\BB)$ whose image by $\pi$ contains $g_n(p_n)$. By mixing of the geodesic flow, for all large enough $n$:

$$
\left| \frac{1}{m_{\M}(\B)} \int_{\B_n} f \, dm_{\M} - \int_{\M}f \, dm_{\M} \right| < \frac{\varepsilon}{4}.  
$$

It thus follows from Proposition \ref{comparison} applied to the special flat measure $m_{\M}$ that

$$
\left|\frac{1}{m_{\M}(B)}\int_{B_n} f \, dm_{\M} - \frac{1}{\xmeasure(\LL_n)} \int_{\LL_n} f \circ \pi \xmeasure \right| < \frac{\varepsilon}{4} .  
$$

Consequently, for large enough $n$, 
$$
\left|\int_{\M} f \, dm_{\M} - \frac{1}{\xmeasure(\LL_n)} \int_{\LL_n} f \circ \pi \xmeasure \right| < \frac{\varepsilon}{2},
$$

and thus

$$
\frac{1}{\xmeasure(\LL_n)} \int_{\LL_n} f \circ \pi \xmeasure > \int_{\M} f \, dm_{\M} - \frac{\varepsilon}{2} > 0.
$$ 
 
This implies that there is a $p \in \pi(\LL_n) \subset W^{uu}(q)$ such that $f(p) >0$. 
\end{proof}

  In order to upgrade Theorem \ref{thm: special case first} to Theorem \ref{thm:
 densehorosphere}, we will need the following result, proven in the appendix
    \ref{appendix: extending cylinders}:
  \begin{thm}[Paul Apisa and Alex Wright]\label{thm: Apisa Wright}
Let  $\M$ be an invariant subvariety and suppose that $q \in \M^{(1)}$
has horizontal saddle connections, but is not horizontally
periodic. Then there is $q' \in W^{uu}(q)$ 
such that all horizontal saddle connections on $M_{q'}$ are longer than
the shortest horizontal saddle connection on $M_q$. If in addition
$M_q$ has no horizontal cylinders, then for any $T>0$  there is $q' \in W^{uu}(q)$
such that the shortest horizontal saddle connection on $M_{q'}$ is
longer than $T$.  
    \end{thm}

\begin{proof}[Proof of Theorem \ref{thm: densehorosphere}]
We repeat the arguments given in the proof of Theorem \ref{thm: special case first}. In that proof, the only place where we used the assumption that $q$ has no horizontal saddle connections, is to ensure the existence of $p_n$ satisfying condition~\eqref{eq: what p must satisfy}. For this, using Proposition \ref{prop: horocycle invariance} and Lemma \ref{nondivergence}, it is enough to show that  there is $q_n' \in W^{uu}(g_{-n}q)$ such that the shortest horizontal saddle connection in $M_{q'_n}$ has length at least one. By our assumption, $M_q$ has no horizontal cylinders and therefore neither does $M_{g_{-n}q}$, and thus we can conclude using the second assertion of Theorem \ref{thm: Apisa Wright}.  
\end{proof}

\begin{rem}\label{remark: leaf closures} 
Theorem \ref{thm: densehorosphere} describes the leaf-closure of cylinder-free surfaces. More generally, if $q \in \M^{(1)}$ is a surface with horizontal cylinders $C_1, \ldots, C_s$, its leaf-closure $\overline{W^{uu}(q)}$ can be described as a properly embedded bundle $\mathcal{B}$ in $\M^{(1)}$, generalizing the discussion in \S \ref{subsec: H(1,1) classification}. We sketch the argument here. Any surface in $q' \in W^{uu}(q)$ also has corresponding cylinders  $C'_1, \ldots, C'_s$ (see the proof of Proposition \ref{prop: same saddle connections}). Let $c'_j, a'_j, \tau'_j$ denote respectively their circumferences, heights, and twists. We have $a'_j = a_j$, that is the heights of the cylinders are the same on the surfaces $M_q$ and $M_{q'}$. Arguing as in the proof of Proposition \ref{prop: properly embedded}, the numbers $c'_j, \tau'_j$ belong to a bounded subset $\Delta \subset \R^{2s}$. Let $\Delta_{q}$ be the subset of $\Delta$ describing cylinders that can arise for $q' \in W^{uu}(q)$. This set is the base of $\mathcal{B}$. The fiber over $\vec{c}, \vec{\tau} \in \Delta_q$ is described as follows.  Let $q' \in W^{uu}(q)$ have cylinders $C'_1, \ldots, C'_s$ whose geometry is prescribed by $\vec{c}, \vec{\tau}$, and  let $q''$ be the surface obtained by removing the cylinders $C'_1, \ldots, C'_s$ and regluing the boundary components to each other by a translation. The translation in each cylinder is chosen so that singularities on opposite sides of a cylinder are not glued to each other. This surface can be alternatively described as the limit $t \to 0$ of the cylinder stretch map $\bar g$, for $g=\left( \begin{matrix} t & s \\ 0 & 1\end{matrix} \right)$, for some $s$  (see Proposition \ref{prop: cylinder deformation}). The surface $q''$ belongs to some invariant subvariety $\M''$, independent of $\vec{c}, \vec{\tau}$, in a lower dimensional stratum. It has no horizontal cylinders, so by Theorem \ref{thm: densehorosphere}, its horosphere is dense in $\M''$. Note that the area of $q''$ is not one and so we apply Theorem \ref{thm: densehorosphere} after rescaling. Thus the fibers of $\mathcal{B}$ are all isomorphic to $\M''$. As $\vec{c}, \vec{\tau}$ leaves compact subsets of $\Delta_q$, the corresponding surfaces have shorter and shorter horizontal saddle connections, and thus $\mathcal{B}$ is properly embedded in $\M$.
\end{rem}

\appendix

\section{Extending saddle connections (Apisa and Wright)}\label{appendix: extending cylinders}  

In this section we give the proof of Theorem \ref{thm: Apisa Wright}. We need some auxiliary statements. A {\em horizontal cylinder} on a translation surface is a cylinder whose core curve is horizontal. We say that a cylinder and a saddle connection are {\em disjoint} if they do not intersect, except perhaps at singular points. Our convention is that cylinders are closed, and thus a cylinder and a saddle connection on one of its boundary components are not considered to be disjoint.

We recall the notion of $\M$-equivalence of cylinders, introduced in \cite{Wright_cylinders}.  Let $\M$ be an invariant subvariety, let $q \in \M$ and let $C_1, C_2$ be two parallel cylinders in $M_q$. The cylinders are called {\em $\M$-parallel} if there is a neighborhood $\mathcal{U}  $ of $q$ in $\M$, such that $C_1, C_2$ remain parallel for all $q' \in \mathcal{U}$. More precisely: 

\begin{itemize}
    \item there is a lift $L$ of $\M$ and open $\mathcal{V} \subset L$ and $\mathcal{U} \subset \M$ such that $q \in \mathcal{U}$, $\pi|_{\mathcal{V}}: \mathcal{V} \to \mathcal{U}$ is a homeomorphism and $\dev$ is injective on $\mathcal{V}$;
    \item for $\q \in \mathcal{V}$ with $q = \pi(\q)$, represented by a marking map $f: S \to M_q$, and for any $\q' \in \mathcal{V}$, represented by $f': S \to M_{q'}$, the sets $f' \circ f^{-1}(C_i), \ i=1,2$ are parallel cylinders on $q' = \pi(\q')$. 
\end{itemize} 
Being $\M$-parallel is clearly an equivalence relation.

For a cylinder $C$ on a translation surface $M$, we denote by $G_C$ the subgroup of $\GL$ fixing the holonomy of the core curve of  $C$. Clearly $G_{C_1} = G_{C_2}$ if $C_1, C_2$ are parallel. If $C_1, \ldots, C_r$ are parallel on $M$ and $g \in G_{C_1}$ then the {\em cylinder surgery corresponding to $g, C_1, \ldots, C_r$} is a modification of the surface $M$ obtained by applying $g$ to the $C_i$ and leaving the complement $M \sm \bigcup_{i=1}^r C_i$ untouched. 
For example if $C$ is horizontal then the elements of $G_C$ are of the form $\left (\begin{matrix} 1 & s \\ 0 & t\end{matrix} \right)$, with $t>0$. The cylinder surgery of such a matrix with $t=1$ consists of {\em cylinder shears} with shear parameter $s$. The cylinder surgery with $s=0$ consists of {\em cylinder stretches} with stretch parameter $t$. By an appropriate conjugation, the definition of cylinder shears and stretches is extended to non-horizontal cylinders.

We have: 
\begin{prop}[Wright, \cite{Wright_cylinders}]\label{prop: cylinder deformation}
For $\M$, any $q \in \M$, and an $\M$-parallel equivalence class of cylinders $C_1, \ldots, C_r$ on $M_q$, if $g \in G_{C_i}$ then the surface  obtained from $M_q$ by cylinder surgery corresponding to $g, C_1, \ldots, C_r$ is also in $\M$. 

\end{prop} 
Suppose $q \in \M$ and $C_1, \ldots, C_r $ are $\M$-parallel cylinders on $M_q$, which are not necessarily a full equivalence class of $\M$-parallel cylinders. 
Let $L$ be a lift of $\M$, let $\q \in L \cap \pi^{-1}(q)$ and let $V \subset H^1 \left( S, \Sigma; \C \right)$ such that $\dev(L)=V.$  
Varying $g \in G_{C_1}$ gives rise to a two dimensional collection (in the previous example, corresponding to possible choices of the parameters $s,t$) of surfaces, obtained from $M_q$
 by cylinder surgery corresponding to $g, C_1, \ldots, C_r$. This collection corresponds to a complex affine line in period coordinates. A generator for this complex line is 
 \begin{equation}\label{eq: sigma defn}
 \sigma_{\{C_i, h_i\}}
 \df \sum_{i=1}^r h_i \gamma_i^*,     
 \end{equation}
 where $h_i$ is the height of $C_i$, $\gamma_i$ is the core curve of $C_i$  and $\gamma_i^*$ is the dual class to $\gamma_i$ in $H^1(S, \Sigma).$ Moving along the line tangent to $\sigma_{\{C_i, h_i\}}$ in $\M $ amounts to performing cylinder shears in each of the $C_i$, and moving along the line tangent to $\mathbf{i} \cdot \sigma_{\{C_i, h_i\}}$ in $\M $ amounts to performing cylinder stretches. 
 
 Below we will be interested in such one-parameter families of deformations that preserve $\M$ and are tangent to $\sigma_{\{C_i, h_i\}} $ as in equation~\eqref{eq: sigma defn}, in which the $\{C_i\}$ might not be a full equivalence class of $\M$-parallel cylinders, and the $h_i$ might not be their heights. Note that for cylinder shears, such surgeries are well-defined for any value of the shear parameter $s$, and for cylinder stretches, they are well-defined as long as $t>0$.

 When $\sigma_{\{C_i, h_i\}} \in V$ for $\q, L, V $ as above we simply say that {\em $\sigma_{\{C_i, h_i\}}$ is contained in the tangent space to $\M$ at $q$.}
 
\begin{prop}\label{prop: provides cylinders}
If $\M$ is an invariant subvariety and $q \in \M$ is not horizontally periodic, then there is a nonempty collection of $\M$-parallel cylinders $C_1, \ldots, C_r$ on $M_q$ which consists of cylinders disjoint from all horizontal saddle connections on $M_q$,
and positive $h_1, \ldots, h_r$ such that the class $\sigma_{\{C_i, h_i\}}$ as in equation~\eqref{eq: sigma defn} is contained in the tangent space to $\M$ at $q.$
Furthermore, there is constant $A_0>0$, depending only on $\M$, so that if $q \in \M$ has no horizontal cylinders then one can choose a collection of cylinders with these properties so that, in addition, the sum of the areas of the cylinders is at least $A_0$. 
\end{prop}

We note that the $h_i$ in Proposition \ref{prop: provides cylinders} might not be the heights of the $C_i.$ The vector $\sigma_{\{C_i, h_i\}}$ is tangent to the line in $\M$ obtained by applying different cylinders shears to each $C_i$. In order to formalize this, 
 for $g_i \in G_{C_i}$ for each $i$, we define the {\em cylinder surgery corresponding to $\{C_i, g_i\}$} to be the modification of $M_q$ obtained by applying $g_i$ to each $C_i$, leaving the complement $M \sm \bigcup_{i=1}^r C_i$ untouched. With this terminology, 
the line tangent to $\sigma_{\{C_i, h_i\}}$ is the collection of surfaces obtained by cylinder surgery corresponding to $\{C_i, g_i\},$ where $s \in \R$ and  $g_i $ performs a cylinder shear with parameter $sh_i$ in $C_i$. Similarly, the line tangent to $\mathbf{i} \cdot \sigma_{\{C_i, h_i\}}$ is the collection of surfaces obtained by cylinder surgery corresponding to $\{C_i, g_i\},$ where $s \in \R$ and  $g_i $ performs a cylinder stretch with parameter $sh_i$ in $C_i$.

\begin{proof}[Proof of Proposition \ref{prop: provides cylinders}]
The validity of both statements is unchanged if we replace the surface $q$ with some surface $q'$  in its horocycle orbit $Uq$. This follows from the facts that the $U$-action is linear in charts,  preserves horizontal saddle connections and maps $\M$-equivalent cylinders to $\M$-equivalent cylinders preserving their area. According to
\cite{SW_calanque}, there is $q_\infty$ in the closure of $Uq$ which is horizontally periodic, and we will see that the required properties hold for all $q'$ sufficiently close to $q_\infty$. Let $M_\infty$ be the underlying surface of ${q_\infty}$, let $C_1, \ldots, C_r$ be the horizontal cylinders on $M_\infty$, and for each $i$, $A_i, c_i, h_i, \gamma_i$ denote respectively the area, circumference, height, and core curve of $C_i$. Here we consider $\gamma_i$ as an element of $H_1(S, \Sigma) $ by using a marking $f: S \to M_\infty$ corresponding to $\q_\infty \in \pi^{-1}(q_\infty).$ By Proposition \ref{prop: cylinder deformation}, $\sigma_{\{C_i, h_i\}}$ belongs to the tangent space of $\M$ at $q.$

For any $\theta_0>0 $ there is a neighborhood $\mathcal{U} = \mathcal{U}(\theta_0)$ of $q_\infty$ in $\M$ such that  if $q' \in \mathcal{U}$ then the underlying surface has $r$ cylinders $C'_1, \ldots, C'_r$ of circumferences $c'_i$ and areas $A'_i$ satisfying

\begin{equation}\label{eq: the bounds}
    c'_i < \bar c \df  
2\max_{i=1, \ldots, r} c_i, \ \ \ \ A'_i
> \underline{A} \df \frac12 \min_{i=1, \ldots, r} A_i ,
\end{equation}
and with directions of core curves in $(-\theta_0, \theta_0)$.

If $C$ is a cylinder and $\sigma$ is a saddle connection on a translation surface, recall that we say that {\em $\sigma$ crosses $C$} if it intersects all the core curves of $C$. Since a cylinder contains no singularities in its interior, if a saddle connection intersects the interior of a cylinder, then it must cross it. Let $s$  be the maximal length of a horizontal saddle connection on $M_q$ and let $\theta_0$ be small enough so that a horizontal segment of length $s$ cannot cross a cylinder of direction $\theta $ satisfying $0 < |\theta| < \theta_0$, with circumference at most $\bar c$ and area at least $\underline{A}$. 

By making $\mathcal{U}$ smaller, so that it is an evenly covered neighborhood of $q_\infty$, we can ensure that 
$\sigma_{\{C'_i, h_i\}}$ belongs to the tangent space of $\M$ at $q'.$ Indeed, if $\mathcal{V}$ is a connected component of $\pi^{-1}(\mathcal{U})$ and $\q', \q_\infty \in \mathcal{V}$ are preimages of $q', q_\infty$ respectively, then the core curves of the cylinders $C'_i, C_i$ map to the same elements $\gamma_i \in H_1(S, \Sigma)$ under the corresponding marking maps, and thus $\sigma_{\{C'_i, h_i\}} = \sigma_{\{C_i, h_i\}}$.

Now suppose that $q' \in Uq \cap \mathcal{U}$, and let $M'$ be the underlying surface. Since $q$ is not horizontally periodic, neither is $q'$. Therefore there is an equivalence class $C_1, \ldots, C_r$ of $\M$-parallel cylinders on $M_\infty$, so that the corresponding cylinders $C'_1, \ldots, C'_r$ are not horizontal cylinders on $M'$, and satisfy the bounds in equation~\eqref{eq: the bounds}. Furthermore the maximal length of a horizontal saddle connection on $M'$ is $s$, since the horocycle flow maps horizontal saddle connections to horizontal saddle connections of the same length. By choice of $\theta_0$, the cylinders in this equivalence class are all disjoint from horizontal saddle connections on $M'$. This proves the first assertion.  

Let $t$ be an upper bound on the number of horizontal cylinders for a surface in $\M$ and let $A_0 \df \frac{1}{2t}$. The argument above works for any collection of $\M$-parallel cylinders $C_1, \ldots, C_r$ which are horizontal on $q_\infty$ and are not horizontal on $q'$. If $q$ has no horizontal cylinders then neither does $q'$, and we can apply the argument with any equivalence class of $\M$-parallel horizontal cylinders $C_1, \ldots, C_r$ on $M_\infty$. One of these classes must have total area at least $\frac{1}{t}$, and thus for $\mathcal{U}$ sufficiently small, the sums of the areas of the corresponding cylinders $C'_1, \ldots, C'_r$ is at least $A_0$. 
\end{proof}

\begin{proof}[Proof of Theorem \ref{thm: Apisa Wright}] We first prove the first assertion. Let $C_1, \ldots, C_r$ be the $\M$-parallel cylinders on $M_q$, and $h_1, \ldots, h_r$ the positive numbers provided by Proposition \ref{prop: provides cylinders}. 
We define $q'$ as $g \circ \varphi (q)$, where $g \in \GL$ and $\varphi$ is a cylinder surgery corresponding to $\{C_i, g_i\}$, and $g_i$ is the cylinder stretch with parameter $sh_i$ for some $s<0.$ 
Although the maps $g, \varphi$ do not preserve the area of the surface, we will choose parameters so that their composition does. Moreover neither of these maps changes the vertical component of the holonomy of any curve. 
The $\GL$-action preserves $\M$, and by Proposition \ref{prop: provides cylinders}, so does $\varphi.$ The map $g$ will increase the length of all horizontal saddle connections on $M_q$, and the cylinder surgery $\varphi$ will not affect their length, since the cylinders $C_i$ are disjoint from the horizontal saddle connections on $q$.

The area of $\varphi (q)$ is smaller than the area of $q$ since all of the cylinders $C_i$ are stretched by a negative parameter. Let $A$  be the sum of the areas of the cylinders $C_1, \ldots, C_r$. Then by choosing the parameter $s$ appropriately, we can arrange so that the area of $\varphi(q)$ is $1-\frac{A}{2}.$ We now set 
\begin{equation}\label{eq: def t}
    t \df \left(1 -\frac{A}{2}\right)^{-1} \ \ \ \text{ and } g \df \left( \begin{matrix} t & 0 \\ 0 & 1 \end{matrix} \right). 
\end{equation}

Then $g$ increases the lengths of horizontal saddle connections by a factor $t>1$, and multiplies the area of $\varphi(q)$ by $t$. This completes the proof of the first assertion.

For the second assertion, we use the second assertion in Proposition \ref{prop: provides cylinders} to choose the cylinders so the sum of their areas satisfies $A \geq A_0$. This ensures that the horizontal saddle connections on $q' \df g \circ \varphi (q)$ are longer than the horizontal saddle connections on $q$ by a factor of at least $t$, where $t>1$ is as in equation~\eqref{eq: def t} and $t-1$ is bounded away from $0$. In light of Proposition \ref{prop: same saddle connections},  $q' \in W^{uu}(q)$ will not have horizontal cylinders either. So we can apply the above argument iteratively, at each stage obtaining surfaces in $W^{uu}(q)$ with longer and longer horizontal saddle connections. Since the lengths of these horizontal  saddle connections grows by a definite amount in each step, after finitely many steps they will all be longer than $T$. 
\end{proof}

\section{Measures on $\HH$ and $\Mod(S,\Sigma)$-invariant measures on $\HHm$}\label{appendix: correspondence principle}
The goal of this section is to prove a result on the correspondence between Radon measures on $\HH$ and $\Mod(S,\Sigma)$-invariant Radon measures on $\HHm$. This result is part of the folklore but we were not able to find a reference; see \cite[Prop. 1.3]{Furstenberg_horocycle} for an analogous result in a restricted setting. 

We state the result in a general setting.
Let $\tilde X$ be a paracompact manifold and $\Gamma$ a discrete group acting properly discontinuously on $\tilde X$. We will write the $\Gamma$-action as an action on the right.  
Let $X=\tilde X/\Gamma$ be the quotient space and $\pi:\tilde X\to X$ the quotient map. 
If $\Gamma$ acts freely then $X$ is a manifold and $\pi$ is a covering map. If $\Gamma$ does not act freely then we can view $X$ as an orbifold and $\pi$ as a regular orbifold covering map (although no knowledge of orbifolds is assumed in this section). 
We do not assume that the action on $\Gamma$ is faithful but, since the action of $\Gamma$ is proper, the subgroup of $\Gamma$ that acts trivially on $\tilde X$ must be finite.

For $\tilde q\in\tilde X$ let $\Gamma(\tilde q)$ be the stabilizer of $\tilde q$ in $\Gamma$.
For any $q \in X$, we define a measure on $\tilde X$ by
\begin{equation}\label{eq: appendix counting measure}
\theta_q \df \sum_{\tilde q \in \pi^{-1}( q)} |\Gamma(\tilde q)|\cdot \delta_{\tilde q},
\end{equation}
where $\delta_{\tilde q}$ is the Dirac mass at $\tilde q$. The measure $\theta_{q}$ is supported on $\pi^{-1}(q)$.  For any $f \in C_c(\tilde X)$ and $\tilde q \in \tilde X$ we have
\begin{equation}\label{eq; other expression}
\int_{\tilde X} f \, d\theta_{\pi(\tilde q)} = \sum_{\gamma \in \Gamma}f(\tilde q \cdot \gamma).
\end{equation}
 It follows from the fact that $\Gamma$ acts properly discontinuously on $\tilde X$ that the sum on the right-hand side is finite.

\begin{defin}\label{def: preimage of measure}
Given a Radon measure $\nu$ on $X$ we define a Radon measure $\tilde{\nu}$ on $\tilde X$, called the {\em pre-image of $\nu$}, by the formula 
\begin{align}\label{liftdef1}
        \int_{\tilde X} f \ d\tilde{\nu} = \int_{X} \left(\int_{\tilde X} f \ d\theta_q \right) \ d\nu(q) \ \ \text{ for any } f \in C_c(\tilde X).
    \end{align}
\end{defin}
Equation~\eqref{liftdef1} defines a unique Radon measure $\tilde{\nu}$ on $\tilde X$ in light of the Riesz Representation Theorem. To see that \eqref{liftdef1} converges, note that the integrand $q \mapsto F(q) 
\df \int_{\tilde{X}} f \ d\theta_q$ is a Borel function, which is supported on the compact set $\pi( \mathrm{supp} \, f),$ and is bounded by $D \, \| f \|_\infty,$ where $D \df \# \{\gamma \in \Gamma: (\mathrm{supp} \, f) \cdot \gamma \cap \mathrm{supp} \, f \neq \emptyset\}  $ is finite since the $\Gamma$-action is properly discontinuous.  

By equation~\eqref{eq; other expression} the measures $\theta_q$ are all $\Gamma$-invariant, and since $\tilde \nu$ is an average of the measures $\theta_q$, we have: 

\begin{lem}\label{lem:invariance}
The measure $\tilde{\nu}$ is  $\Gamma$-invariant.
\end{lem}

The following converse can be understood as a disintegration theorem for $\Gamma$-invariant Radon measures on $\tilde X$.

\begin{prop}\label{prop: correspondence}
    Let $m$ be a $\Gamma$-invariant Radon measure on $\tilde X$. There is a unique Radon measure $\mu$ on $X$ such that $m$ is the pre-image of $\mu$.
\end{prop}

We call $\mu$ the {\em image} of $m$.

\begin{proof}
Let $m$ be given. We are claiming the existence of a Radon measure $\mu$ so that 
equation~\eqref{liftdef1} holds (with $\tilde{\nu}, \nu$ replaced with $m, \mu$). 
The idea of the proof is to build $\mu$ on small neighborhoods using the fact that $\pi$ is an orbifold cover. This will be made rigorous using a partition of unity. Let $\tilde q \in \tilde X$ and let $\Gamma(\tilde q)$ be the stabilizer of $\tilde q$ in $\Gamma$. Since $\Gamma$ acts properly discontinuously on $\tilde X$,  $\Gamma(\tilde q)$ is finite, and there is connected $\Gamma(\tilde q)$-invariant neighborhood $\V$ of $\tilde q$ and a neighborhood $\U$ of $\pi(\tilde q)$ such that $\pi$ induces a homeomorphism $\V / \Gamma(\tilde q) \to \U$, and 
    \begin{equation*}
        \pi^{-1}(\U) =  \bigsqcup_{\gamma \in \Gamma(\tilde q) \backslash \Gamma} \V \cdot \gamma,
    \end{equation*}
where $\gamma$ ranges over a set of coset representatives, and where the sets $\mathcal{V}\cdot \gamma$ are  disjoint.
 We say that such a $\U\subset X$ is {\em evenly covered (in the orbifold sense)}. 
 
 Let $(\U_i)_{i \in I}$ be a locally finite cover of $X$ by evenly covered neighborhoods. 
 Such a cover exists by the paracompactness of $\tilde X$ and the considerations above. 
 For each $i \in I$ choose a connected component $\V_i$ of $\pi^{-1}(\U_i)$. Denote by $\Gamma_i=\Gamma(\V_i)$ the stabilizer of $\V_i$ in $\Gamma$. 
 Let $(\rho_i)_{i\in I}$ be a partition of unity subordinate to the cover $(\U_i)_{i \in I}$ and define a Radon measure $\mu$ on $X$ (by using the Riesz Representation Theorem) such that for any $f \in C_c(X)$,
    \begin{align}\label{disintegration}
    \int_{X} f \ d\mu = \sum_{i \in I} \frac{1}{|\Gamma_i|} \int_{\V_i} (\rho_i f) \circ \pi \, dm.
    \end{align}

    We want to show that the measure $\mu$ satisfies equation~\eqref{liftdef1}. 
    We claim first that for any $\tilde q \in \tilde X$, there is a neighborhood $\V$ around $\tilde q$ such that equation~\eqref{liftdef1} holds for any $f \in C_c(\tilde X)$ with support contained in $\V$. 
    Indeed, let $\tilde q \in \tilde X$ and let $\V$ be a neighborhood of $\tilde q$ small enough so that for any $i \in I$, $\V$ intersects at most one connected component of $\pi^{-1}(\U_i)$. 
    This is possible since the cover by the $\U_i$ is locally finite. Let $J = \{i \in I : \pi(\V) \cap \U_i \neq \emptyset \}$ and for $j \in J$, let $\gamma_j \in \Gamma$ be such that $\V_j \cdot \gamma_j \cap \V \neq \emptyset$. 
    By the assumption on $\mathcal{V}$, the coset  $\Gamma_j \cdot \gamma_j$ is uniquely determined. 
    Let $f \in C_c(\tilde X)$ with support contained in $\V$. We compute:
 \begin{equation}
 \begin{split}\label{computation}
        \int_{X} \Big( \int_{\tilde X} f \ d\theta_q \Big) \ d\mu(q) 
        &= \sum_{i \in I} \frac{1}{|\Gamma_i|} \int_{\V_i} \sum_{\gamma \in \Gamma} \rho_i(\pi(\tilde q) )f(\tilde q \cdot \gamma) \ dm(\tilde q) \\
        &= \sum_{j \in J} \frac{1}{|\Gamma_j|} \int_{\V_j \cdot \gamma_j} \sum_{\gamma \in \Gamma_j} \rho_j(\pi(\tilde q)) f(\tilde q \cdot \gamma) \ dm(\tilde q) \\
        &= \sum_{j \in J} \frac{1}{|\Gamma_j|} \int_{\tilde X} \sum_{\gamma \in       \Gamma_j} \rho_j(\pi(\tilde q)) f(\tilde q) \ dm(\tilde q) \\
        &= \sum_{j \in J} \int_{\tilde X} \rho_j(\pi(\tilde q)) f(\tilde q) \ dm(\tilde q) = \int_{\tilde{X}} f \, dm.
        \end{split}
    \end{equation}
    
    Now let  $f$ be an arbitrary compactly supported continuous function and let $K$ denote its support. 
    Using a covering argument and the computation above, we can find finitely many $(\W_i)_i$ that cover $K$ and such that equation~\eqref{disintegration} holds for continuous functions with support contained in $\W_i$. Let $(\psi_i)_i$ be a partition of unity associated with this cover.
    We can write $f = \sum_i \psi_if$. By construction, each of the $\psi_if$ has support contained in $\W_i$ and the result follows by equation~\eqref{computation} and the linearity of the integral.  
    
    To prove uniqueness of the measure $\mu$, we proceed as follows. Let $\mu_1$ and $\mu_2$ be two Radon measures on $X$ that satisfy equation~\eqref{liftdef1} (with $m,\mu_i $ instead of $\tilde{\nu}, \nu $).
    Let $f \in C_c(X)$ be a compactly supported continuous function whose support is contained in an evenly covered neighborhood $\U$ and let $\V$ be a connected component of $\pi^{-1}(\U)$. 
    We denote by $\Gamma(\V)$ the stabilizer in $\Gamma$ of $\V$. Let $h$ be the function on $\tilde X$ that is equal to  $f \circ \pi$ on $\V$ and vanishes outside of $\V$. 
    Since the support of $f$ is contained in $\U$, the function $h$ is continuous and has compact support. Furthermore, it is easy to see that for any $q \in X$, $\int_{\tilde X} h \ d\theta_q = |\Gamma(\V)| f(q)$. This implies
    \begin{align*}
        \int_{X} f \ d\mu_1 = \frac{1}{|\Gamma(\V)|}  \int_{\tilde X} h \ dm =\int_{X} f \ d\mu_2.
    \end{align*}
    To deal with the case when $f$ is an arbitrary function of compact support we appeal once more to existence of partitions of unity. 
\end{proof}

Let $G$ be a group acting on $\tilde{X}$ so that the action commutes with the action of $\Gamma$.
The group $G$ induces an action on $X$ so that $\pi$ is $G$-equivariant.

\begin{prop} A Radon measure $\mu$ on $X$ is invariant under $g\in G$ if and only if its pre-image $\tilde\mu$ is invariant under the action of $g$ on $\tilde X$.
\end{prop}

\begin{proof} 
We start by proving two formulas showing the naturality of the pre-image construction.
\begin{claim}\label{claim1}
$g_*(\theta_q)=\theta_{g(q)}.$
\end{claim}

\begin{align}
g_*(\theta_q)& 
 =\sum_{\pi(\tilde q)=q}|\Gamma(\tilde q)|\cdot \delta_{g(\tilde q)}\nonumber\\
            & =\sum_{\pi(g(\tilde q))=g(q)}|\Gamma(\tilde q)|\cdot \delta_{g(\tilde q)}\label{cleq6}\\
        &=\sum_{\pi(g(\tilde q))=g(q)}|\Gamma(g(\tilde q))|\cdot \delta_{g(\tilde q)}=\theta_{g(q)}. \label{cc5} 
\end{align}
In line~\eqref{cleq6} we used the fact that $\pi(g(\tilde q))=g(\pi(\tilde q))=g(q)$. In line~\eqref{cc5} we used the fact
that the $\Gamma$ action commutes with $g$, which implies that $|\Gamma(g(\tilde q))|=|\Gamma(\tilde q)|$.

\begin{claim}\label{claim2}
$g_*(\tilde\mu)=\widetilde{g_*(\mu)}$.
\end{claim}

It suffices to show that both measures assign the same integrals to continuous functions of compact support on $\tilde X$. Let $f$ be such a function.
\begin{align}\nonumber
\int_{\tilde X}f\, d\widetilde{g_*(\nu)}& =\int_{X} \Big(\int_{\tilde X} f \ d\theta_q \Big) \ dg_*(\nu)(q)\\
& =\int_{X} \Big(\int_{\tilde X} f \ d\theta_{g(q)} \Big) d\nu(q) =\int_{X} \Big(\int_{\tilde X} f \ dg_*(\theta_q) \Big) d\nu(q) \label{c2eq2}\\
& =\int_{X} \Big(\int_{\tilde X} f\circ g\,  d\theta_q \Big) d\nu(q) \nonumber\\
& =\int_{\tilde X}f\circ g\, d\tilde\nu =\int_{\tilde X}f\, dg_*(\tilde\nu).  \label{c2eq5}
\end{align}
In line~\eqref{c2eq2} we used Claim \ref{claim1}. In line~\eqref{c2eq5} we used the definition of the pre-image applied to the function of compact support $f\circ g$.

Using these formulas we now prove the Proposition.
If $g_*(\mu)=\mu$ then $g_*(\tilde\mu)=\widetilde{g_*(\mu)}=\tilde\mu$ by Claim \ref{claim2}. So $\tilde\mu$  is invariant under the action of $g$ on $\tilde X$.
If $g_*(\tilde\mu)=\tilde\mu$ then $\widetilde{g_*(\mu)}=g_*(\tilde\mu)=\tilde\mu$ so the pre-images of $g_*(\mu)$ and $\mu$ are equal. 
It follows from the uniqueness assertion in Proposition \ref{prop: correspondence} that $g_*(\mu)=\mu$.
\end{proof}

\bibliographystyle{alpha}
\bibliography{horospherical.bib}

\begin{thebibliography}{EMM22}

\bibitem[AG10]{avila2010small}
A.~Avila and S.~Gou\"{e}zel.
\newblock Small eigenvalues of the {L}aplacian for algebraic measures in moduli
  space, and mixing properties of the {T}eichm{\"u}ller flow.
\newblock {\em Ann. of Math. (2)}, 178:385--442, 2010.

\bibitem[AGY06]{AGY}
Artur Avila, S\'{e}bastien Gou\"{e}zel, and Jean-Christophe Yoccoz.
\newblock Exponential mixing for the {T}eichm\"{u}ller flow.
\newblock {\em Publ. Math. Inst. Hautes \'{E}tudes Sci.}, (104):143--211, 2006.

\bibitem[BSW22]{BSW}
Matt Bainbridge, John Smillie, and Barak Weiss.
\newblock Horocycle dynamics: new invariants and eigenform loci in the stratum
  {$\mathcal{H}(1,1)$}.
\newblock {\em Mem. Amer. Math. Soc.}, 280(1384):v+100, 2022.

\bibitem[Cal04]{Calta}
Kariane Calta.
\newblock Veech surfaces and complete periodicity in genus two.
\newblock {\em J. Amer. Math. Soc.}, 17(4):871--908, 2004.

\bibitem[CC03]{Candel_Conlon}
Alberto Candel and Lawrence Conlon.
\newblock {\em Foliations. {II}}, volume~60 of {\em Graduate Studies in
  Mathematics}.
\newblock American Mathematical Society, Providence, RI, 2003.

\bibitem[CSW20]{chaika2020tremors}
Jon Chaika, John Smillie, and Barak Weiss.
\newblock Tremors and horocycle dynamics on the moduli space of translation
  surfaces.
\newblock {\em arXiv preprint arXiv:2004.04027}, 2020.

\bibitem[Dan78]{dani1978invariant}
S.~G. Dani.
\newblock Invariant measures of horospherical flows on noncompact homogeneous
  spaces.
\newblock {\em Invent. Math.}, 47(2):101--138, 1978.

\bibitem[Dan81]{dani1981invariant}
SG~Dani.
\newblock Invariant measures and minimal sets of horospherical flows.
\newblock {\em Inventiones mathematicae}, 64(2):357--385, 1981.

\bibitem[EM01]{eskin_masur_2001}
Alex Eskin and Howard Masur.
\newblock Asymptotic formulas on flat surfaces.
\newblock {\em Ergodic Theory and Dynamical Systems}, 21(2):443–478, 2001.

\bibitem[EM18]{EM}
Alex Eskin and Maryam Mirzakhani.
\newblock Invariant and stationary measures for the $\mathrm{SL}(2,\mathbb{R})$
  action on moduli space.
\newblock {\em Publications mathématiques de l'Institut des hautes études
  scientifiques}, 127(1):95--324, 2018.

\bibitem[EMM15]{EMM}
Alex Eskin, Maryam Mirzakhani, and Amir Mohammadi.
\newblock Isolation, equidistribution, and orbit closures for the
  $\mathrm{SL}(2,\mathbb{R})$ action on moduli space.
\newblock {\em Ann. of Math. (2)}, 182(2):673--721, 2015.

\bibitem[EMM22]{EMM_JEMS}
Alex Eskin, Maryam Mirzakhani, and Amir Mohammadi.
\newblock Effective counting of simple closed geodesics on hyperbolic surfaces.
\newblock {\em J. Eur. Math. Soc. (JEMS)}, 24(9):3059--3108, 2022.

\bibitem[EW11]{Einsiedler_Ward}
Manfred Einsiedler and Thomas Ward.
\newblock {\em Ergodic theory with a view towards number theory}, volume 259 of
  {\em Graduate Texts in Mathematics}.
\newblock Springer-Verlag London, Ltd., London, 2011.

\bibitem[FM14]{Forni-Matheus_survey}
Giovanni Forni and Carlos Matheus.
\newblock Introduction to {T}eichm\"{u}ller theory and its applications to
  dynamics of interval exchange transformations, flows on surfaces and
  billiards.
\newblock {\em J. Mod. Dyn.}, 8(3-4):271--436, 2014.

\bibitem[For21]{Forni_density_one}
Giovanni Forni.
\newblock Limits of geodesic push-forwards of horocycle invariant measures.
\newblock {\em Ergodic Theory Dyn. Syst.}, 41(9):2782--2804, 2021.

\bibitem[Fur73]{Furstenberg_horocycle}
Harry Furstenberg.
\newblock The unique ergodicity of the horocycle flow.
\newblock Recent {Advances} topol. {Dynamics}, {Proc}. {Conf}. topol.
  {Dynamics} {Yale} {Univ}. 1972, {Lect}. {Notes} {Math}. 318, 95-115 (1973).,
  1973.

\bibitem[Gol]{Goldman}
William Goldman.
\newblock Geometric structures on manifolds.
\newblock AMS Graduate Studies in Mathematics, accepted for publication.

\bibitem[Ham09]{hamenstadt2009invariant}
Ursula Hamenst\"adt.
\newblock Invariant {R}adon measures on measured lamination space.
\newblock {\em Inventiones mathematicae}, 176(2):223--273, 2009.

\bibitem[Kre85]{Krengel}
Ulrich Krengel.
\newblock {\em Ergodic theorems}, volume~6 of {\em De Gruyter Studies in
  Mathematics}.
\newblock Walter de Gruyter \& Co., Berlin, 1985.
\newblock With a supplement by Antoine Brunel.

\bibitem[Lee13]{lee2012smooth}
John~M. Lee.
\newblock {\em Introduction to smooth manifolds}, volume 218 of {\em Graduate
  Texts in Mathematics}.
\newblock Springer, New York, second edition, 2013.

\bibitem[LM08]{lindenstrauss2008ergodic}
Elon Lindenstrauss and Maryam Mirzakhani.
\newblock Ergodic theory of the space of measured laminations.
\newblock {\em International mathematics research notices}, 2008(1), 2008.

\bibitem[Mar04]{Margulis_thesis}
Grigoriy~A. Margulis.
\newblock {\em On some aspects of the theory of {A}nosov systems}.
\newblock Springer Monographs in Mathematics. Springer-Verlag, Berlin, 2004.
\newblock With a survey by Richard Sharp: Periodic orbits of hyperbolic flows,
  Translated from the Russian by Valentina Vladimirovna Szulikowska.

\bibitem[McM07]{McMullen_SL2}
Curtis~T. McMullen.
\newblock Dynamics of {${\rm SL}_2(\Bbb R)$} over moduli space in genus two.
\newblock {\em Ann. of Math. (2)}, 165(2):397--456, 2007.

\bibitem[MT02]{MT}
Howard Masur and Serge Tabachnikov.
\newblock Rational billiards and flat structures.
\newblock In {\em Handbook of dynamical systems, {V}ol. 1{A}}, pages
  1015--1089. North-Holland, Amsterdam, 2002.

\bibitem[MW02]{MW}
Yair Minsky and Barak Weiss.
\newblock Nondivergence of horocyclic flows on moduli space.
\newblock {\em Journal für die reine und angewandte Mathematik},
  2002(552):131--177, 2002.

\bibitem[MW14]{MW2}
Yair Minsky and Barak Weiss.
\newblock Cohomology classes represented by measured foliations, and {M}ahler's
  question for interval exchanges.
\newblock {\em Ann. Sci. \'{E}c. Norm. Sup\'{e}r. (4)}, 47(2):245--284, 2014.

\bibitem[SW04]{SW_calanque}
John Smillie and Barak Weiss.
\newblock Minimal sets for flows on moduli space.
\newblock {\em Israel J. Math.}, 142:249--260, 2004.

\bibitem[SY]{SY}
John Smillie and Florent Ygouf.
\newblock Geometric structures and orbit closures in moduli spaces of
  translation surfaces.
\newblock In preparation.

\bibitem[Vee86]{veech1986the}
William~A Veech.
\newblock The {T}eichm\"uller geodesic flow.
\newblock {\em Ann. of Math. (2)}, 124(3):441--530, 1986.

\bibitem[Wri15a]{Wright_cylinders}
Alex Wright.
\newblock Cylinder deformations in orbit closures of translation surfaces.
\newblock {\em Geom. Topol.}, 19(1):413--438, 2015.

\bibitem[Wri15b]{Wright_survey}
Alex Wright.
\newblock Translation surfaces and their orbit closures: an introduction for a
  broad audience.
\newblock {\em EMS Surv. Math. Sci.}, 2(1):63--108, 2015.

\bibitem[Zor06]{Zorich_survey}
Anton Zorich.
\newblock Flat surfaces.
\newblock In {\em Frontiers in number theory, physics, and geometry. {I}},
  pages 437--583. Springer, Berlin, 2006.

\end{thebibliography}

\end{document}